\documentclass[12pt]{article}
\usepackage[margin=1in]{geometry}
\usepackage{color}
\usepackage{amsmath,amsthm,amssymb,mathtools,bbm,mathrsfs}
\usepackage[authoryear,longnamesfirst]{natbib}
\usepackage{array}
\usepackage{enumitem}
\usepackage{textcase}

\usepackage[breaklinks=true,hidelinks]{hyperref}

\numberwithin{equation}{section}
\allowdisplaybreaks[4]

\newtheoremstyle{plain2}{\topsep}{2\topsep}{\itshape}
{0pt}{\bfseries}{.}{.5em}{}
\newtheoremstyle{definition2}{\topsep}{2\topsep}{}
{0pt}{\bfseries}{.}{.5em}{}

\theoremstyle{plain2}
\newtheorem{theorem}{Theorem}[section]
\newtheorem{proposition}[theorem]{Proposition}
\newtheorem{lemma}[theorem]{Lemma}
\newtheorem{corollary}[theorem]{Corollary}

\theoremstyle{definition2}

\newtheorem{remark}[theorem]{Remark}

\makeatletter

\renewcommand{\cite}{\citet}

\def\^#1{\ifmmode {\mathaccent"705E #1} \else {\accent94 #1} \fi}
\def\~#1{\ifmmode {\mathaccent"707E #1} \else {\accent"7E #1} \fi}

\def\*#1{#1^\ast}
\edef\-#1{\noexpand\ifmmode {\noexpand\bar{#1}} \noexpand\else \-#1\noexpand\fi}
\def\>#1{\vec{#1}}
\def\.#1{\dot{#1}}

\def\atop{\@@atop}
\def\%#1{\mathcal{#1}}

\let\original@left\left
\let\original@right\right
\renewcommand{\left}{\mathopen{}\mathclose\bgroup\original@left}
\renewcommand{\right}{\aftergroup\egroup\original@right}

\renewcommand{\phi}{\varphi}
\newcommand{\eps}{\varepsilon}
\newcommand{\D}{\Delta}
\newcommand{\eq}{\eqref}

\newcommand{\dtv}{\mathop{d_{\mathrm{TV}}}\mathopen{}}

\newcommand{\dloc}[1][]{\mathop{d_{\mathrm{loc}}^{\,#1}}\mathopen{}}

\newcommand{\bigo}{\mathop{\mathrm{{}O}}\mathopen{}}

\def\ER{Erd\H{o}s--R\'enyi}
\newcommand{\I}{\mathop{{}\mathrm{I}}}
\newcommand{\Po}{\PIP}

\newcommand{\TP}{\mathop{\mathrm{TP}}}

\newcommand{\Bi}{\mathop{\mathrm{Bi}}}

\newcommand{\Be}{\mathop{\mathrm{Be}}}

\newcommand{\N}{\mathrm{N}}

\newcommand{\IE}{\mathop{{}\mathbb{E}}\mathopen{}}
\newcommand{\IP}{\mathop{{}\mathbb{P}}\mathopen{}}
\newcommand{\Var}{\mathop{\mathrm{Var}}}
\newcommand{\Cov}{\mathop{\mathrm{Cov}}}

\newcommand{\law}{\mathscr{L}}

\newcommand{\IZ}{\mathbb{Z}}
\newcommand{\IR}{\mathbb{R}}

\def\be#1{\begin{equation*}#1\end{equation*}}
\def\ben#1{\begin{equation}#1\end{equation}}
\def\bes#1{\begin{equation*}\begin{split}#1\end{split}\end{equation*}}
\def\besn#1{\begin{equation}\begin{split}#1\end{split}\end{equation}}
\def\bg#1{\begin{gather*}#1\end{gather*}}
\def\bgn#1{\begin{gather}#1\end{gather}}

\def\ba#1{\begin{align*}#1\end{align*}}
\def\ban#1{\begin{align}#1\end{align}}
\def\given{\mskip 0.5mu plus 0.25mu\vert\mskip 0.5mu plus 0.15mu}
\newcounter{@bracketlevel}
\def\@bracketfactory#1#2#3#4#5#6{
\expandafter\def\csname#1\endcsname##1{%
\addtocounter{@bracketlevel}{1}%
\global\expandafter\let\csname @middummy\alph{@bracketlevel}\endcsname\given%
\global\def\given{\mskip#5\csname#4\endcsname\vert\mskip#6}\csname#4l\endcsname#2##1\csname#4r\endcsname#3%
\global\expandafter\let\expandafter\given\csname @middummy\alph{@bracketlevel}\endcsname
\addtocounter{@bracketlevel}{-1}}%
}
\def\bracketfactory#1#2#3{%
\@bracketfactory{#1}{#2}{#3}{relax}{0.5mu plus 0.25mu}{0.5mu plus 0.15mu}
\@bracketfactory{b#1}{#2}{#3}{big}{1mu plus 0.25mu minus 0.25mu}{0.6mu plus 0.15mu minus 0.15mu}
\@bracketfactory{bb#1}{#2}{#3}{Big}{2.4mu plus 0.8mu minus 0.8mu}{1.8mu plus 0.6mu minus 0.6mu}
\@bracketfactory{bbb#1}{#2}{#3}{bigg}{3.2mu plus 1mu minus 1mu}{2.4mu plus 0.75mu minus 0.75mu}
\@bracketfactory{bbbb#1}{#2}{#3}{Bigg}{4mu plus 1mu minus 1mu}{3mu plus 0.75mu minus 0.75mu}
}
\bracketfactory{klg}{\lbrace}{\rbrace}
\bracketfactory{klr}{(}{)}
\bracketfactory{kle}{[}{]}
\bracketfactory{abs}{\lvert}{\rvert}
\bracketfactory{norm}{\Vert}{\Vert}
\bracketfactory{floor}{\lfloor}{\rfloor}
\bracketfactory{ceil}{\lceil}{\rceil}
\bracketfactory{angle}{\langle}{\rangle}

\newcommand{\displayrelskip}{\mkern6mu}
\def\leq{\mathchoice%
{\mathrel{\mkern6mu\mathchar"0436\mkern6mu}}%
{\mathchar"3436}{\mathchar"3436}{\mathchar"3436}}
\let\le\leq
\def\geq{\mathchoice%
{\mathrel{\mkern6mu\mathchar"043E\mkern6mu}}%
{\mathchar"343E}{\mathchar"343E}{\mathchar"343E}}
\let\ge\geq
\AtBeginDocument{
\begingroup
\lccode`~61
\lowercase{\endgroup
\def~{\mathchoice%
{\mathrel{\mkern6mu\mathchar61\mkern6mu}}%
{\mathrel{\mathchar61}}{\mathrel{\mathchar61}}{\mathrel{\mathchar61}}}
}
\mathcode61"8000
}
\def\Def{\mathchoice%
{\mathrel{\mkern4mu\vcentcolon\mathrel{\mathchar61}\mkern8mu}}%
{\mathrel{\vcentcolon\mathrel{\mathchar61}}}%
{\mathrel{\vcentcolon\mathrel{\mathchar61}}}%
{\mathrel{\vcentcolon\mathrel{\mathchar61}}}}
\newcommand{\ts}{\textstyle}

\newcount\minute
\newcount\hour
\newcount\hourMins
\def\now{%
\minute=\time%
\hour=\time \divide \hour by 60%
\hourMins=\hour \multiply\hourMins by 60%
\advance\minute by -\hourMins%
\zeroPadTwo{\the\hour}:\zeroPadTwo{\the\minute}%
}
\def\zeroPadTwo#1{\ifnum #1<10 0\fi#1}

\renewcommand\section{\@startsection {section}{1}{\z@}%
{-3.5ex \@plus -1ex \@minus -.2ex}%
{1.3ex \@plus.2ex}%sm
{\center\small\sc\mathversion{bold}\MakeTextUppercase}}

\def\subsection#1{\@startsection {subsection}{2}{0pt}%
{-3.5ex \@plus -1ex \@minus -.2ex}%
{1ex \@plus.2ex}%
{\bf\mathversion{bold}}{#1}}

\def\subsubsection#1{\@startsection{subsubsection}{3}{0pt}%
{\medskipamount}%
{-10pt}%
{\normalsize\itshape}{\kern-2.2ex. #1.}}

\def\blfootnote{\xdef\@thefnmark{}\@footnotetext}

\makeatother

\def\sp#1{^{(#1)}}

\newcommand{\PIP}{\mathcal{P}}
\newcommand\cG{\mathcal{G}}
\newcommand\cN{\mathcal{N}}
\def\giv{\,|\,}
\def\a{\alpha}
\def\b{\beta}
\def\s{\sigma}
\def\d{\delta}
\def\g{\gamma}
\def\k{\kappa}
\def\la{\lambda}
\def\law{{\mathcal L}}
\def\cV{{\mathcal V}}
\def\cE{{\mathcal E}}
\def\ignore#1{}
\def\tW{{\widetilde W}}
\def\hn{{\hat n}}
\def\hW{{\widehat W}}
\def\hcV{{\widehat {\cV}}}
\def\tE{{\widetilde E}}
\def\tcG{{\widetilde {\cG}}}
\def\ppi{\rho}
\def\un{^{(n)}}

\def\cI{{\mathcal I}}
\def\cJ{{\mathcal J}}
\def\ha{{\hat a}}
\def\half{\tfrac12}
\def\quarter{\tfrac14}

\def\bF{{\overline F}}
\def\h{\eta}
\def\ess{\g}
\def\U{\Upsilon}
\def\tG{{\widetilde G}}
\def\tR{{\widetilde R}}
\def\l{\lambda}

\begin{document}

\title{\sc\bf\large\MakeUppercase{Error bounds in local limit theorems using Stein's method}}
\author{\sc A.\ D.\ Barbour, Adrian R\"ollin and Nathan Ross}
\date{\it Universit\"at Z\"urich, National University of Singapore and University~of~Melbourne}
\maketitle

\begin{abstract} 
We provide a general result for bounding the difference between 
point probabilities of integer supported distributions and the translated Poisson distribution, 
a convenient alternative to the discretized normal. 
We illustrate our theorem in the context of the Hoeffding combinatorial central limit theorem with integer valued summands,
of the number of isolated vertices in an \ER\ random graph, 
and of the Curie--Weiss model of magnetism, where we provide optimal or near optimal rates of convergence in the local limit metric. 
In the Hoeffding example, even the discrete normal approximation bounds
seem to be new. The general result follows from Stein's method, and requires a new bound
on the Stein solution for the Poisson distribution, which is of general interest.

\end{abstract}

% \noindent\textbf{Keywords: } 

\section{Introduction}\label{sec1}

The local limit theorem for general sums~$W$ of independent integer valued random variables
began with the seminal work of \cite{Esseen1945}, and is now
well understood \cite[Chapter~VII]{Petrov1975}. For sums of dependent
random variables, however, much less is known.  A key idea, introduced by \cite{McDonald1979},
is to prove local theorems by using a combination of the corresponding (global) central
limit theorem, together with an {\it a priori\/}
estimate of the smoothness of the distribution~$\law(W)$ being approximated.  \cite{Rollin2005} used
this strategy, combined with Stein's method, to develop a systematic approach to
approximation by the discrete normal distribution, not only locally, but also globally with 
respect to the total variation distance. 

In both \cite{McDonald1979} and \cite{Rollin2005}, the smoothness estimates are derived by
finding a suitable large collection of conditionally independent Bernoulli random variables
embedded in the construction of~$W$.  
In \cite{Rollin2015}, a fundamentally different
technique was discovered, which is instead based on finding a suitable exchangeable pair 
in the spirit of 
\cite[Chapter~I, Lemma~3]{Stein1986}.  They combined it with Landau-Kolmogorov inequalities
to give local limit approximations in a variety of examples, but often with less than
optimal rates. In this paper, we use Stein's method and the smoothness approach
to give a general local limit approximation theorem for settings in which dependence can 
be described in terms of an (approximate) Stein coupling as given in \cite{Chen2010}.
This formulation is very flexible, and includes exchangeable
pair, local dependence and size-bias settings as particular instances. 
In the examples that we consider, our approach yields bounds that,
when expressed as functions of~$\s^2 \Def  \Var W$, are no worse than a log factor from the
optimal rate of~$\bigo(\s^{-2})$.
Our general bound
is expressed in terms of quantities that typically arise when using Stein's method
in the central limit context.  As a result, we are able to give bounds for the total variation
error in discrete normal approximation as well as the local limit bounds with no
extra effort.

\subsection{Translated Poisson distribution}
As in  \cite{Rollin2007}, we use \emph{translated Poisson} distributions as approximating family instead of discretised normal
distributions --- Lemma~\ref{lem1} justifies this to the accuracy of interest to us. We say that the random variable~$Z$ has the translated Poisson distribution and
write~$Z\sim\TP(\mu, \sigma^2)$ if~$Z-s \sim \Po(\sigma^2+\gamma)$, where
\ben{\label{1}
    s \Def \floor{\mu-\sigma^2},\quad \gamma \Def \mu - \sigma^2 - \floor{\mu-\sigma^2},
}
and where~$\Po(\l)$ denotes the Poisson distribution with mean~$\l$.
Note that~$\IE Z = \mu$ and~$\sigma^2\leq \Var Z \leq \sigma^2+1$. 
The translated Poisson distribution is a Poisson distribution, but translated by an integer 
chosen so that both its mean and variance closely match prescribed values~$\mu$ and~$\s^2$.
The following lemma say that the translated Poisson distribution is an appropriate substitute for the discretized
normal distribution. Its proof follows easily from the classical local central limit theorem with error.

\begin{lemma}\label{lem1}
There exists a constant~$C>0$ such that, for all~$\mu\in\IR$ and~$\sigma^2\ge1$,
\be{
  \sup_{n\in \IZ} \bbbabs{ \TP(\mu,\sigma^2)\{n\} -
      \frac{1}{\sqrt{2\pi\sigma^2}}\exp\bbbklr{-\frac{(n-\mu)^2}{2\sigma^2}}} \leq \frac{C}{\sigma^{2}}.
}
\end{lemma}

\noindent We also note some basic properties of the translated Poisson distributions. Define the following
``smoothness" measure of an integer valued distribution,
\ben{\label{2}
  S_l(\law(W)) \Def \sup_{h:\norm{h}\leq 1} \abs{\IE \Delta^l h(W)}, \qquad \ts l\ge1,
}
where~$\Delta$ denotes the first difference operator~$\Delta g(k)\Def g(k+1)-g(k)$. 
Variations of the smoothing terms~\eq{2} frequently appear 
in integer supported distributional approximation results; 
see, for example, \cite{Barbour1999}, \cite{Goldstein2006}, \cite{Rollin2008a} and \cite{Fang2014}.

\noindent The next result shows the typical smoothness expected for approximately discretized normal 
distributions. It is shown in \cite[Lemma~4.1]{Rollin2015}.
\begin{lemma}\label{lem2} For each~$k\geq 1$ there exists a constant~$C(k)$ such that, for all~$\mu\in\IR$ and~$\sigma^2\ge1$,
\bgn{
    S_k(\TP(\mu,\sigma^2)) \leq \frac{C(k)}{\sigma^{k}}. \label{3}
}
\end{lemma}

\subsection{Stein couplings}

Our approximations are designed for random variables~$W$ that form part of a {\it Stein coupling\/}.
Following \cite{Chen2010}, we say that the random variables~$(W,W',G,R)$ with~$\IE W=\mu$ form an 
\emph{approximate Stein coupling} if
\ben{
   \IE[G (f(W')-f(W))] = \IE[(W-\mu)f(W)]+\IE [Rf(W)],\label{4}
}
for all~$f$ such that the expectations exist. If~$R=0$ almost surely,  we call~$(W,W',G)$ a Stein coupling. Some examples of 
Stein couplings well used in Stein's method are the following:
\begin{description}[leftmargin=0em,rightmargin=0.5em]
\item[Local dependence.] Let~$W=\sum_{i=1}^n X_i$, with~$ \IE X_i = \mu_i$ 
for~$1\leq i\leq n$. Suppose that, for each~$i$, there is
$A_i\subset\{1,\dots,n\}$ such that~$X_i$ is independent of~$(X_j)_{j\not\in A_i}$. Then, for~$I$ a random index, 
uniformly distributed on~$\{1,\ldots, n\}$ and independent of $(X_i)_{i=1}^n$,
\ben{\label{locdepstncoup}
  (W, W', G) \Def \Bigl(W, W-\sum_{j\in A_I} X_j, -n (X_I-\mu_I) \Bigr),
}
is a Stein coupling.

\item[Size bias.] If~$W^s$ has the size bias distribution of~$W$ and~$\IE W=\mu$, then
\be{
  (W,W',G) \Def (W, W^s, \mu)
}
is a Stein coupling.

\item[Exchangeable pairs.] If~$(W, W')$ is an exchangeable pair satisfying the linearity condition
\ben{
   \IE[W'-W|W] = -a(W-\mu)+a R, \label{5}
}
then 
\be{
   ( W,W',G, R) \Def \bbbklr{W, W', \frac{W'-W}{2a},  R }
}
is an approximate Stein coupling. 

\item[Exchangeable pairs, one-sided version.] If~$(W, W')$ is an exchangeable pair that satisfies~\eq{5},
then 
\ben{\label{6}
   ( W,W',G, R) \Def \bbbklr{W, W', \frac{W'-W}{a}\I[{\ts W'-W>0}],  R }
}
is an approximate Stein coupling. 

\end{description}

\noindent Note also that, for ~$(W,W',G,R)$  an approximate Stein coupling,
\ben{
   \IE [G(W'-W)] = \sigma^2+\IE[R(W-\mu)], \label{7}
}
which can be seen by taking~$f(x)=x$ and~$f(x)=1$ (to find~$\IE R=0$) in the defining relation~\eq{4}.
In particular, if~$R=0$ almost surely, then~$\IE [G(W'-W)] = \s^2$.

\section{Main results and applications}\label{sec2}

We bound the error in the approximation by the translated Poisson distribution of the distributions~$\law(W)$
of integer valued random variables with finite variances that can be represented as the~$W$ in an 
(approximate) Stein coupling.
Our bounds are expressed in terms of the moments of~$W$ and of expectations involving the
quantities~$G$ and~$D \Def  W'-W$, and the conditional smoothness coefficients~$S_l(\law(W|\cal{F}))$ for some appropriate
associated sigma-field~$\cal{F}$.  
Exchangeable pairs, size-biasing, and local dependence appear ubiquitously 
when using Stein's method for distributional approximation and concentration inequalities,
so (1) many of the terms appearing in our bound can be fruitfully bounded using well-established techniques,
and (2) new techniques developed here for bounding commonly appearing terms will prove useful 
in other applications of Stein's method.

\subsection{An abstract bound}

In order to express the accuracy of translated Poisson approximation, we define the \emph{total variation metric} as
\be{
  \dtv\bklr{\law(X), \law(Y)} \Def \sup_{A \subseteq \IZ}|\IP[X\in A] - \IP[Y\in A]|,
 }
as well as a metric to capture the local differences as
\be{
   \dloc\bklr{\law(X), \law(Y)} \Def \sup_{a \in \IZ}|\IP[{\ts X=a}] - \IP[{\ts Y=a}]|.
}

We can now state our main general approximation result, which is proved in Section~\ref{sec:genres}.
\begin{theorem}\label{thm1}
Let~$(W,W',G,R)$ be an approximate Stein coupling with~$W$ and~$W'$ integer valued, ~$\IE W=\mu$ and~$\Var(W)=\sigma^2$.
Set~$D\Def W'-W$, and let~$\%F_1$ and~$\%F_2$ be sigma-algebras such that\/~$W$ is~$\%F_1$-measurable and such that\/~$(G,D)$ is ~$\%F_2$-measurable. Define  
\be{
  \Psi \Def \babs{\IE [GD |\% F_1]-\IE [GD]}, \qquad \U \Def \IE\bkle{\abs{G D(D-1)}S_2(\law\klr{W\given\%F_2})}.
}
Then
\ben{
  \dtv\bklr{\law(W),\TP(\mu,\sigma^2)}  
   \leq \frac{\IE\Psi}{\sigma^2} 
      + \frac{2\sqrt{\IE R^2}}{\sigma}
        + \frac{2(\U + 1)}{\s}, \label{8}
}
and
\besn{
  \dloc\bklr{\law(W),\TP(\mu,\sigma^2)} 
  &\leq \frac{\IE \Psi}{\sigma^{3} \sqrt{2e}} +\frac{\IE[\Psi \abs{W-\mu}]}{ \sigma^{4}} 
      +\sup_{a\in\IZ}\frac{\IE \{\Psi \I[W=a]\}}{\sigma^2}\\
  &\qquad\quad+ \frac{\sqrt{\IE R^2}}{\sigma^2}\Bigl(2+\frac{1}{\sqrt{2e}}
                +\sigma \sup_{a\in\IZ} \IP({\ts W=a}) \Bigr)  + \frac{2(\U + 1)}{\s^2}. \label{9}
}
\end{theorem}

Note that we have distinguished $\Psi$ and $\U$ as the significant quantities in the bound, but that $\Psi$ is random while $\U$ is not.

\begin{remark}\label{rem1} If~$(W,W')$ is an exchangeable pair such that~\eq{5} holds, and if we assume in addition that~$D\in\{-1,0,+1\}$, then we can use the coupling \eq{6}; in this case, $\Psi$ and~$\U$ simplify to 
\ben{\label{eq1}
  \Psi = \frac{1}{a}\babs{\IP({\ts D=1}|\%F_1) -\IP({\ts D=1})}
  \qquad\text{and}\qquad \U = 0.
}
\end{remark}

\subsection{Towards a concrete bound}

The bounds in Theorem~\ref{thm1} are still rather abstract, and it may not be obvious how to handle the individual terms in concrete applications. We now show that, by making a natural additional assumption, the terms appearing in \eq{8} and \eq{9} can be made more manageable. 

To this end, we assume that, for some~$\kappa>0$, for some integer~$k\geq 0$ and for some non-negative random variable~$T$, we have
\ben{\label{10}
  \Psi \leq \sigma  \k\sum_{j=0}^k \left(\frac{\abs{W-\mu}}{\sigma}\right)^j+ T.
}
This assumption appears naturally in many applications, and it is worthwhile emphasising that it is weaker than similar conditions appearing in the literature around Stein's method, such as Condition (3.3) in Theorem 3.1 of \cite{Chen2013a} or the condition in Theorem~3.11 of \cite{Rollin2007}. We now have the following easy corollary of Theorem~\ref{thm1}.

\begin{corollary}\label{cor1} Under the conditions of Theorem~\ref{thm1} and assuming in addition \eq{10}, we have
\ban{
  &\dtv\bklr{\law(W),\TP(\mu,\sigma^2)}  
   \leq \frac{\k}{\sigma}\sum_{j=0}^{k} \frac{\IE \abs{W-\mu}^j}{\sigma^j}
    + \frac{\IE T}{\s^2}
      + \frac{ 2\sqrt{\IE R^2}}{\sigma}
        + \frac{2(\U + 1)}{\s}. \label{11}
} 
and
\ban{
  &\dloc\bklr{\law(W),\TP(\mu,\sigma^2)} \notag \\
  &\qquad \leq \frac{2\k}{\sigma^2}\sum_{j=0}^{k+1} \frac{\IE \abs{W-\mu}^j}{\sigma^j}\label{12}
        + \frac{\k}{\sigma^2}\,\sup_{a\in\IZ} \left( \IP({\ts W=a})\sum_{j=0}^k\frac{\abs{a-\mu}^j}{\sigma^{j-1}}\right)\\
  &\qquad\qquad +\frac{2 \sqrt{\IE T^2}}{\sigma^3} +\frac{\sup_{a\in\IZ}  \IE\kle{T\I[{\ts W=a}] }}{\sigma^2} 
                   \label{13}\\
  &\qquad\qquad + \frac{\sqrt{\IE R^2}}{\sigma^2}\left(3+\sigma \sup_{a\in\IZ} \IP({\ts W=a}) \right) \label{14} \\
  & \qquad\qquad + \frac{2(\U + 1)}{\s^2}. \label{15}
}
\end{corollary}

Assumption~\eq{10} is always satisfied by taking~$T=\Psi$ and an empty sum, so its real use
is if~$T$ is more easily managed than~$\Psi$ --- for instance, if~$T=0$ almost surely.
In what follows, we take assumption~\eq{10} to be satisfied, and consider
the bounds \eq{12}--\eq{15} in turn. We tacitly assume throughout the following discussion
that we have a sequence of integer valued random variables~$W = W_m$ with means~$\mu=\mu_m$ and whose variances 
$\s^2 = \s^2_m$ grow
to infinity with~$m$;  order estimates are to be understood  as~$m\to\infty$, and the dependence on~$m$ is
suppressed in the notation.

First, we expect~$\IE \{\abs{\s^{-1}(W-\mu)}^{k+1}\}$ to be bounded, so that the first term of~\eq{12} is of our 
target order~$\bigo(\s^{-2})$. For the second term of~\eq{12}, we have the following lemma.

\begin{lemma}\label{lem3}
Write~$\d \Def  \dtv\bklr{\law(W),\TP(\mu,\sigma^2)}$ and assume that for some~$1/2 < \a \le 1$, we have  
\be{
   \d = \bigo(\sigma^{-\a}) \quad \mbox{ and } \quad S_2(\law(W))=\bigo(\sigma^{-1-\a}).
}
Then, for any~$\ell \geq j \geq1$, 
\be{
\sup_{a\in\IZ} \IP({\ts W=a})\frac{\abs{a-\mu}^j}{\sigma^{j-1}}\leq \bigo(1) +  
     \IE\Bigl\{ \Bigl(\frac{|W-\mu|}\s\Bigr)^\ell \Bigr\} \sigma^{1/2 + \a - (2\a-1)\ell/(2j)},
}
and
\be{
\sup_{a\in\IZ}\IP({\ts W=a})=\bigo(\sigma^{-1}).
}

\end{lemma}

\begin{proof}
Recall the definitions~$s =  \floor{\mu-\s^2}$ and~$\g = \mu - \s^2 - s$, and let~$\PIP_\lambda(\cdot):=\Po(\lambda)\{\cdot\}$.
First note that
\be{
   \sup_{a\in\IZ} \IP({\ts W=a}) \leq \d + \bigo(\sigma^{-1}),
}
which follows easily from the definition of total variation, and because~$\sup_{a\in\IZ} \PIP_\lambda(a) = \bigo(\lambda^{-1/2})$ as~$\lambda\to \infty$, where
$\la = \s^2 + \g$.

For the first assertion, we 
first bound~$\IP(W=a)$ for~$a$
``near"~$\mu$. Note that the second assertion follows from~\eq{eq:t} below, used in this argument, and the last sentence of the previous paragraph. Now, \cite[Theorem~2.2\emph{(i)} with~$l=2,~m=1$, and Lemma~3.1]{Rollin2015} 
implies that for some constant~$C$,
\be{
	\dloc(\law(W), \TP(\mu,\sigma^2)) 
		\leq C \dtv(\law(W),\TP(\mu,\sigma^s))^{1/2} \left(S_2(\law(W))+S_2(\TP(\mu,\sigma^2))\right)^{1/2},
}
which, with~\eq{3} and the hypotheses of the lemma, implies
\be{
   \dloc(\law(W),\TP(\mu,\sigma^2)) = \bigo(\sigma^{-1/2-\a}).
}
Hence
\ben{\label{eq:t}
    \abs{\IP({\ts W=a})-\PIP_{\s^2+\g}(a-s)} = \bigo(\sigma^{-1/2-\a}),
}
so that
\be{
   \IP({\ts W=a})\frac{\abs{a-\mu}^j}{\sigma^{j-1}}
     \leq \frac{\abs{a-\mu}^j}{\sigma^{j-1/2+\a}}+ \PIP_{\s^2+\g}(a-s)\frac{\abs{a-\mu}^j}{\sigma^{j-1}}.
}
Combining this inequality with the observation that 
\be{
  \sup_{\lambda\geq 1} \sup_{r\in\IZ} \PIP_\lambda(r) \frac{\abs{r-\lambda}^j}{\lambda^{(j-1)/2}}< \infty,
}
and noting that, for~$r = a-s$ and~$\la = \s^2+\g$, 
\[
    r - \la = a-(\mu-\s^2 - \g) - (\s^2+\g) = a - \mu,
\]
it follows that ~$\IP(W=a)\{\s^{-(j-1)}\abs{a-\mu}^j\} = \bigo(1)$ for~$\s^{-1}\abs{a - \mu}\leq \sigma^{(2\a-1)/(2j)}$. 

For values of~$a$ ``far'' from~$\mu$, that is, $\s^{-1}\abs{a - \mu} > \sigma^{(2\a-1)/(2j)}$, use Markov's inequality to give
\bes{
 \frac{\abs{a-\mu}^j}{\sigma^{j-1}}\IP({\ts W=a})&\leq 
                \frac{\s\abs{a-\mu}^j}{\sigma^{j}}\IP\left(\frac{|W-\mu|}\s{\ts{} \geq {}}\frac{|a-\mu|}\s\right) \\
 	&\leq \s \Bigl|\frac{a-\mu}\s\Bigr|^{-(l-j)} 
             \IE\Bigl\{\Bigl(\frac{|W-\mu|}\s\Bigr)^\ell\Bigr\}\\
   &\leq \s  \IE \Bigl\{\Bigl(\frac{|W-\mu|}\s\Bigr)^\ell\Bigr\} \s^{(2\a-1)(j-\ell)/2j},
}
concluding the proof.
\end{proof}

\begin{remark}\label{rem2}
Thus, if~\eq{10} and the hypotheses of Lemma~\ref{lem3} are satisfied, 
and if~$\IE \abs{W-\mu}^{K}=\bigo(\sigma^{K})$ for some~$K \ge k(1 + 2\a)/(2\alpha-1)$, with~$\a$ as in Lemma~\ref{lem3}, 
then the second term of~\eq{12} is of order~$\bigo(\s^{-2})$.
\end{remark}

We next show that, if~$T$ is concentrated around zero, then the two terms of~\eq{13} can be 
suitably bounded.

\begin{lemma}\label{lem4}
Suppose that the non-negative random variable~$T= T_m$ satisfies
\ben{\label{16}
    \IE\bigl( \s^{-1}T\I[{\ts \s^{-1}T \ge t}]\bigr) \leq \eps(t),\ \qquad\ts t\ge1, 
}
for some $\eps(t)$ with~$\int_1^\infty \eps(t)\,dt < K < \infty$, and~$K$ is the same for all~$m$.
Then~$\IE T^2=\bigo(\sigma^2)$.

If~\eq{16} is satisfied, then for any~$k\in\IZ$ and~$t \geq 1$,
$$
    \IE\left[T\I[{\ts W=k}]\right]\leq \s \eps(t) + t \sigma \sup_{a\in\IZ}\IP({\ts W=a}).
$$
\end{lemma} 

\begin{proof}
By a standard calculation, $\IE T^2 \leq \s^2\bigl(1 + \int_1^\infty \eps(t)\,dt\bigr)$. For the second assertion, 
note that
\be{
  \IE\left[T\I[{\ts W=k}]\right]\leq \IE \{T\I[{\ts T \geq t} \sigma]\} + t\sigma \IP({\ts W=k}).
} 
The former term is bounded by~$\s\eps(t)$ and the latter by $t \sigma \sup_{a\in\IZ}\IP({\ts W=a})$.
\end{proof}

\begin{remark}
For example, suppose that~$\sup_{a\in\IZ}\IP({\ts W=a})= \bigo(\s^{-1})$.  Then if~$\eps(t) = 0$ for all~$t \ge t_0$, for some~$t_0 < \infty$,  
we have a bound for~\eq{13} of the ideal order~$\bigo(\s^{-2})$.
If, for some constant~$c$, we have~$\eps(t) \le e^{-t^2/2c}$, then the choice~$t = \sqrt{2c\log\s}$
gives~$\sup_k \IE\left[T\I[W=k]\right] = \bigo(\sqrt{\log\s})$, and a bound for~\eq{13} of order
$\bigo(\s^{-2}\sqrt{\log\s})$.  If~$\eps(t) \le e^{-t/c}$, then the choice~$t = c\log\s$ gives a bound 
for~\eq{13} of order~$\bigo(\s^{-2}\log\s)$. 
Note also that, under the conditions of Lemma~\ref{lem3},
\[
    \IE T\IP[{\ts W=k}] \leq \sqrt{\IE T^2}\IP[{\ts W=k}] = \bigo(1),
\]
and so, in~\eq{13}, $\sup_a \IE[T\I[{\ts W=a}]]$ can be replaced by
\[
     \sup_a |\Cov(T,\I[{\ts W=a}])| + \bigo(1 ).
\]
\end{remark}

For the remaining terms in Corollary~\ref{cor1}, if~$\IE R^2=\bigo(1)$, then it is easy to see that~\eq{14} is of order~$\bigo(\sigma^{-1} \sup_{a\in\IZ}\IP({\ts W=a}) +\sigma^{-2})$. This leaves~$\U$,
which is handled using the methods discussed in \cite{Rollin2015} and illustrated in the applications below.  
We collect the results above in the following corollary.

\begin{corollary}\label{17}
Assume the notation and hypotheses of Theorem~\ref{thm1} and suppose that~\eq{10} is satisfied for some choice of~$\kappa$, $k$ and~$T$.
\begin{enumerate}[label=\textrm{(\roman*)}]
\item If ~$\IE\{(\s^{-1}|W-\mu|)^{k}\} = \bigo(1)$, ~$\IE T = \bigo(\s)$ and~$\IE R^2=\bigo(1)$,  then it follows that
\[
     \delta\Def \dtv\bklr{\law(W),\TP(\mu,\sigma^2)} = \bigo(\s^{-1}(1 + \U)).
\]
\item\label{cor2} If~$\IE R^2=\bigo(1)$ and, for some~$1/2 < \a \le 1$, 
\begin{enumerate}[label=(\arabic*),leftmargin=3em]
\item\label{condi1}~$\delta = \bigo(\sigma^{-\a})$ and~$S_2(\law(W)) = \bigo(\s^{-1-\a})$;
\item\label{condii1} ~$\IE\{(\s^{-1}|W-\mu|)^{K}\} = \bigo(1)$ for some~$K \ge k(1 + 2\a)/(2\alpha-1)$;
\item\label{condiii1}~$T/\s$ is almost surely\ uniformly bounded,
\end{enumerate}
then 
\ben{\label{18}
   \dloc\bklr{\law(W),\TP(\mu,\sigma^2)} = \bigo(\s^{-2}(\U + 1)).
}
If~\ref{condiii1} is replaced by
\begin{enumerate}[label=(\arabic*a),start=3,leftmargin=3em]
\item\label{condiiia1}~$\sup_a\abs{\Cov(T,\I[W=a])}$ is bounded, 
\end{enumerate}
then~\eq{18} still holds.  If~\ref{condiii1} is replaced by
\begin{enumerate}[label=(\arabic*b),start=3,leftmargin=3em]
\item\label{condiiib1}~$\eps(t)$ of Lemma~\ref{lem4} has an exponential tail, 
\end{enumerate}
then the term~$\bigo(\s^{-2}(\U + 1))$ has to be replaced by~$\bigo(\s^{-2}(\U + 1)\log\s)$ in
the bound in~\eq{18}.  
\end{enumerate}
\end{corollary}

\noindent  Note that the value of~$\a$ used in~\emph{\ref{condi1}} and~\emph{\ref{condii1}} does not appear in the error estimate; the assumptions
are there to ensure that enough moments of~$\s^{-1}|W-\mu|$ are finite.  However, if the largest~$\a$
for which~$S_2(\law(W)) = \bigo(\s^{-1-\a})$ is such that~$\a < 1$, then, even in the ideal case in which 
$|GD(D-1)| = \bigo(\s^2)$ almost surely, the quantity~$\U$ is only guaranteed to be of order~$O(\s^{1-\a})$,
yielding a bound in~\eq{18} of order~$\bigo(\s^{-1-\a})$,
and not of the ideal order~$\bigo(\s^{-2})$.

\begin{remark}[Sums of independent random variables]
If~$W_n=\sum_{i=1}^n X_i$,  where the $X_i$ are independent integer valued random variables 
such that $\sum_{i=1}^n \IE|X_i - \IE X_i|^3 = \bigo(\s_n^2)$, and which satisfy 
an aperiodicity condition, then Theorems~4 and~5 in Chapter VII of \cite{Petrov1975} imply that the error made 
in the local limit approximation by the discrete normal (and hence the translated Poisson) is of best order~$\s_n^{-2}$.
If we assume the somewhat stronger aperiodicity assumption, that $\s_n^{-2}\sum_{i=1}^n(1-\dtv(\law(X_i),\law(X_i+1))$ 
is bounded away from zero,  then
it follows that $S_2(\law(W_n))=\bigo(\sigma_n^{-2})$ and that
$\Upsilon$ is also of the correct order for good rates, so that
the only problem term, in the decomposition~\eq{10}, is $\sup_{k\in\IZ}\IE\left[T\I[{\ts W=k}]\right]$.
Using our approach, in conjunction with the local dependence Stein coupling at~\eq{locdepstncoup} with $A_i=\{i\}$,
we deduce a~$T=\Psi$ of the form 
\be{
\bbabs{\sum_{i=1}^n X_i (X_i-\mu_i)- \sigma_n^2},
}
and, as in Corollary~\ref{17}, this together with Lemma~\ref{lem4} leads to bounds that depend strongly 
on the tail behaviour of $X_i$. For example, if the $X_i$ have finite $(2j)$th moment for some $j\geq 2$, 
then by H\"older's and Rosenthal's inequalities, 
\be{
     \sup_{k\in\IZ}\IE\left[T\I[{\ts W=k}]\right]
   \leq \left(\IE T^{j}\right)^{1/j} \sup_{k\in\IZ}\left(\IP(\ts W=k)\right)^{(j-1)/j} =\bigo\left(n^{-1/(2j)}\right),
} 
with a constant depending on $j$, implying an upper bound on the local metric of sub-optimal order $n^{-1+1/(2j)}$. 
% Even if the square of the summands have finite exponential moment (e.g., if the $X_i$ are bounded), then our bound is of 
% suboptimal order $\sqrt{\log(n)}/n$. 
Thus a direct application of our approach, which can be effective in much more 
challenging applications,  is sub-optimal in this classical case.  As it happens, a small modification
of the proof of Lemma~\ref{lem8} below, adding and subtracting $GD\Delta f(W')$ rather than $GD\Delta f(W)$
after~(\ref{29}),
eliminates the problem term, and leads to an approximation error of the same asymptotic order as
that given in \cite[Theorems~4 and~5 in Chapter VII]{Petrov1975}, albeit under our stronger aperiodicity
assumption. This is essentially the approach taken by \cite[Theorem~2.1]{Rollin2008a}. 
\end{remark}

\subsection{Hoeffding permutation statistic}

Let~$(a_{i j})_{1\leq i,j\leq n}$ be an array of integers, and define
\be{
   W \Def \sum_{i=1}^n a_{i \ppi_i},
}
where~$\ppi$ is a uniformly chosen random permutation.  Defining 
\bg{
   a_{i+} \Def \sum_{j=1}^n a_{ij},
   \qquad a_{+j} \Def \sum_{i=1}^n a_{ij},
    \qquad a_{++} \Def \sum_{i,j=1}^n a_{ij},\\
       \ha_{ij} \Def a_{ij}-\frac{a_{i+}}{n}-\frac{a_{+j}}{n}+\frac{a_{++}}{n^2},
}
we have
\ben{\label{19}
  \mu \Def \IE W = \frac{1}{n} a_{++} \qquad \text{and}\qquad 
           \s^2 \Def \Var W  = \frac{1}{n-1} \sum_{i,j} \ha_{i j}^2. 
}
We are interested in the accuracy of local approximation to~$\law(W)$  by~$\TP(\mu,\s^2)$.

Central limit theorems for~$W$ have a long history going back to 
 \cite{Wald1944} and \cite{Hoeffding1951}. More recent refinements obtaining Berry-Esseen error bounds 
under various conditions on the matrix~$a$ were derived by  \cite{Bolthausen1984}, \cite{Goldstein2005a}, \cite{Chen2015};
see references of the last for an up to date history.

Our main results are in terms of asymptotic rates as~$n\to\infty$ for a sequence of such matrices~$a\un$, assuming 
that, for suitable positive constants~$A_1,\a_0,\a_1$, and~$\a_2$,
\begin{itemize}
 \item {\sl Assumption A1.}~$\max_{1\le i,j\le n}|a_{ij}\un| \leq A_1 < \infty$ and 
     ~$n^{-1}(\s \un)^2 \ge (\a_0 A_1)^2 > 0$  for all~$n$.
       \item  {\sl Assumption A2.} There exists a set~$\cI \Def  \bigl\{\{i_{l1},i_{l2}\},\,1\le l \le n_1\bigr\}$ 
of~$n_1 \ge \a_1 n$ disjoint 
pairs of indices in~$[n]$ such that, for~$\{i_1,i_2\} \in \cI$, there exists a set~$\cJ(i_1,i_2)$
of ~$n_2 \ge \a_2 n^2$ pairs (clearly {\em not\/} disjoint) of indices~$\{j_1,j_2\}$ such that
\[
     |a_{i_1,j_1} + a_{i_2,j_2} - a_{i_1,j_2} - a_{i_2,j_1}| = 1.
\]
\end{itemize}
We also assume without loss that ~$|\mu| \le n/2$, by replacing~$a_{ij}$ by~$a_{ij}+m$ for all~$i,j$, for a suitably chosen integer~$m$.
Assumption A1 is a standard simplifying assumption when studying the Hoeffding permutation statistic and something like
Assumption A2 is necessary to ensure that~$W$ is not concentrated on a sub-lattice of~$\IZ$. For instance, if all 
the~$a_{ij}$ are even, the distribution of~$W$ lies on
the lattice~$2\IZ$, and then~$S_2(\law(W)) = 4$; so some additional conditions on the matrix~$a$ are needed
to ensure smoothness. Our methods still apply if either of these assumptions are weakened, but at the cost of worse bounds or greater technicality. 

Our main result is as follows.

\begin{theorem}
Let~$a\un$ be a sequence of matrices satisfying Assumptions~A1 and~A2 and let~$W=W_n$ be the
Hoeffding permutation statistic defined above. Then,
\ba{
\dtv\left(\law(W),\TP\bklr{\IE W, \Var(W)}\right) &= \bigo\left(\sigma^{-1}\right),\\
\dloc\left(\law(W),\TP\bklr{\IE W, \Var(W)}\right) &= \bigo\left(\frac{\sqrt{\log(\sigma)}}{\sigma^2}\right).
}
\end{theorem}

\subsection{Isolated vertices in the \ER\ random graph}

We show a local limit bound with optimal rate
for~$W$ defined to be the number 
of isolated vertices in an \ER\ graph on~$n$ vertices with edge probability~$p\sim \lambda/n$ for some $\lambda>0$. Note that
\be{
   \mu \Def \IE W = n (1-p)^{n-1};\quad  \sigma^2 \Def \Var(W) = n(1-p)^{n-1}\bkle{1+(np-1)(1-p)^{n-2}},
}
and so in the regime~$p\sim\lambda/n$, we have~$\mu \sim ne^{-\la}$ and~$\sigma^2 \sim ne^{-\la}\{1 + (\la-1)e^{-\la}\}$ are of strict order~$n$. 

Studying degree and subgraph count statistics to understand the structure of \ER\ graphs has a long history, and is still an active area to this day; e.g., \cite{Krokowski2017} and \cite{Rollin2017}.
A number of works derive central limit theorems with error rates for isolated degrees.
Error rates for smooth test function metrics are provided by
 \cite{Barbour1989} and \cite{Kordecki1990}; for
Kolmogorov distance by \cite{Goldstein2013a}; and for total variation
distance (to a discretized normal) by \cite{Fang2014}.
We show the following optimal local limit theorem that 
strengthens the rate provided in \cite{Rollin2015}.
\begin{theorem}
Let~$W=W_n$ be the number of isolated vertices in an \ER\ graph on~$n$ vertices with edge probability 
$p\sim \lambda/n$. Then
\be{
\dloc\left(\law(W),\TP\bklr{\IE W, \Var(W)}\right) = \bigo\left(\frac{\sqrt{\log(\sigma)}}{\sigma^2}\right).
}
\end{theorem}

\subsection{Magnetization in the Curie--Weiss model}

The Curie--Weiss model on~$n$ sites is given by a Gibbs measure on~$\{-1,+1\}^n$ having parameters~$\beta>0$ 
and~$h\in\IR$. The random vector~$S=(S_1,\ldots, S_n)\in\{-1,+1\}^n$ has this distribution if 
\ben{\label{20}
    \IP\bklr{S=(s_1,\ldots,s_n)} = Z_{\beta,h}^{-1}\exp\left\{\frac{\beta}{n} 
                          \sum_{1\leq i< j\leq n} s_i s_j + h \sum_{i=1}^n s_i\right\}.
}
The \emph{magnetization}~$W=\sum_{i=1}^n S_i$ of the system has been the object of intense study over the last forty 
years or more; see~\cite[IV.4.]{Ellis2006}. By symmetry, we only need consider ~$h\geq0$.

We first state the law of large numbers for~$W/n$, which relies on the following equation for fixed~$\beta>0$, $h\geq0$:
\ben{\label{21}
        m = \tanh(\beta m+h).
}
For~$h> 0$, there is only one {\em positive\/} solution~$m_h$ satisfying~\eq{21}.
If~$h=0$ and~$0<\beta<1$, $m_0 = 0$ is the only solution to~\eq{21}.

 \begin{lemma} (\cite[Theorem~IV.4.1]{Ellis2006})
If~$S$ is distributed as~\eq{20} for some~$h>0$ and~$\beta>0$ or for~$h=0$ and~$0<\beta<1$, 
and if~$W=\sum_{i=1}^n S_i$, then as~$n\to\infty$,
\be{
   \frac{W}{n}\displayrelskip\stackrel{prob}{\longrightarrow}\displayrelskip m_h.
}
\end{lemma}

\noindent We then have the following distributional convergence result from \cite[Theorem~2.2]{Ellis1980}.
\begin{theorem}
If~$S$ is distributed as~\eq{20} for some~$h>0$ and~$\beta>0$ or for~$h=0$ and~$0<\beta<1$, 
and if~$W=\sum_{i=1}^n S_i$, then as~$n\to\infty$,
\be{
	\law\left(\frac{W-nm_h}{\sqrt{n}}\right)
	\displayrelskip\to\displayrelskip\N\left(0, \frac{1-m_h^2}{1-\beta+\beta m_h^2}\right).
}
\end{theorem}

\noindent Above the critical temperature,  a convergence rate of order~$\bigo(n^{-1/2})$ in Kolmogorov 
distance is a consequence 
of \cite[Theorem~3 and pp.~602--605]{Barbour1980}; see also \cite{Chatterjee2011a} and~\cite{Eichelsbacher2010}.
Concentration inequalities are derived in \cite{Chatterjee2007} and~\cite{Chatterjee2010}; 
total variation and local limit bounds (that are weaker than those obtained below) are given in~\cite{Rollin2015}.
Note also that, for~$\mu_n\Def \IE W$ and~$\sigma_n^2\Def \Var(W)$, ~$\mu_n \sim n m_h$ and
$\sigma_n^2 \sim n (1-m_h^2)/(1-\beta+\beta m_h^2)$ as~$n\to\infty$.

Our main result for the magnetization is a sharp rate of convergence in the local limit metric.
Note that~$W$ sits on a lattice of span~$2$, so we ultimately shift and scale to put it on~$\{0,\ldots,n\}$.
\begin{theorem}
Let~$S$ be distributed as~\eq{20} for some~$h>0$ and~$\beta>0$ or for~$h=0$ and~$0<\beta<1$, 
$W=W_n=\sum_{i=1}^n S_i$, and~$\tW \Def  (W+\tfrac12\{1-(-1)^n\})/2$. Then in the notation above,
\ba{
\dtv\left(\law\left(\tW\right), 
   	\TP\left(\frac{nm_h}{2},\frac{n(1-m_h^2)}{4(1-\beta+\beta m_h^2)}\right)\right)
		&=\bigo(\s_n^{-1})=\bigo(n^{-1/2}), \\
	\dloc\left(\law\left(\tW\right), 
                         \TP\left(\frac{nm_h}{2},\frac{n(1-m_h^2)}{4(1-\beta+\beta m_h^2)}\right)\right)
		&=\bigo(\s_n^{-2})=\bigo(n^{-1}).
}
\end{theorem}

The remainder of the paper is devoted to proofs of the results above. 
Theorem~\ref{thm1} is proved in the next section, and application statements are proved in Section~\ref{sec4}.

\section{Proof of Theorem~\ref{thm1}}\label{sec:genres}

\subsection{Preliminaries}

To express the accuracy of approximation by a translated Poisson distribution using Stein's method,
we need the solutions~$(g_A)_{A\subset \IZ^+}$ of the Poisson Stein equation
\ben{
    \la \Delta g_A(i)- (i-\la) g_A(i) = \I[i\in A]-\Po(\la)\{A\},\quad {\ts i \ge 0}; 
    \qquad g_A(i) = 0, \quad {\ts i \leq 0}.
                 \label{22}
}
For approximation by~$\TP(\mu,\s^2)$,  defining~$s$ and~$\g$ as in~\eq{1}, we take~$\la \Def  \s^2 + \g$ and 
define~$f_A\colon\IZ\to\IR$ by~$f_A(i)\Def g_A(i-s)$. Where there is no likelihood of confusion, we 
write~$f_a$ for~$f_{\{a\}}$. Then the following representations of the accuracy of translated Poisson 
approximation to an integer valued random variable~$W$ were shown in \cite[(3.18)]{Rollin2007}.

\begin{proposition}\label{lem5}
Let~$W$ be an integer valued random variable with mean~$\mu$ and variance~$\sigma^2$, and let~$f_A$ be defined as above. Then
\ba{
  \dtv\bklr{\law(W), \TP(\mu,\sigma^2)} 
   &\leq\sup_{A\subseteq\IZ^+} 
               \abs{\IE \sigma^2 \Delta f_A(W) - (W-\mu) f_A(W)} +2 \sigma^{-2}, \\
   \dloc\bklr{\law(W), \TP(\mu,\sigma^2)}
   &\leq \sup_{a\in\IZ^+} 
             \abs{\IE \sigma^2 \Delta f_a(W) - (W-\mu) f_a(W)} +2 \sigma^{-2}.
}
\end{proposition}
These inequalities form the basis of our approximations and leveraging them requires detailed understanding
of the functions~$f_{A}$. Though much is known about these functions due to their role in Stein's method for Poisson approximation
 (see, for example, \cite{Barbour1992}), our results require new, finer properties potentially of interest in other Poisson approximation settings; see Lemma~\ref{prop1}.

\subsection{Properties of the solutions of the Poisson Stein equation}\label{sec3}

We first review the known bounds on~$g_A$. 

\begin{lemma}\label{lem6}(\cite{Barbour1992}) 
Let~$A\subseteq\IZ^+$, and let~$g_A$ be as in~\eq{22}. Then 
\ba{
   \norm{g_A}\leq \frac{1}{\lambda^{1/2}} \quad \mbox{ and } \quad \norm{\Delta g_A}
                  \leq \frac{1-e^{-\la}}{\la} \leq \frac{1}{\la}.
}
If~$A=\{a\}$ for some~$a \in \IZ^+$, then
\be{
     \norm{g_{a}}\leq \frac{1}{\lambda}.
}
\end{lemma}

\noindent Note that~$f_A$ satisfies the same bounds as does~$g_A$, but with~$\lambda=\sigma^2+\gamma$.

The bound on~$\norm{g_a}$ is smaller than the general bound on~$\norm{g_A}$ for~$A\subset\IZ^+$
by a factor of~$\la^{-1/2}$.  This suggests that the same might be true for a bound on~$\norm{\D g_a}$,
but it is not the case: In fact, $\D g_a(a)$ is typically comparable to~$\la^{-1}$ and is not of
order~$O(\la^{-3/2})$, as might have been hoped.  In the remainder of this section, we establish non-uniform bounds
on~$|\D g_a(k)|$, showing that~$|\D g_a(k)|$ is nonetheless `typically' of order~$O(\la^{-3/2})$.
This enables us to make the sharper local limit approximations of the paper.
 
Let~$U_j:=\{0,\ldots, j-1\}$ and~$\PIP_\lambda(\cdot):=\Po(\lambda)\{\cdot\}$. Then the solution~$g_a$ of~\eq{22}
with~$A = \{a\}$ can be written as 
\besn{\label{23}
  g_a(k) = \lambda^{-k}e^{\lambda}(k-1)!\, \begin{cases}
     \PIP_\lambda(a) \PIP_\lambda(U_k^c), &k\geq a+1 \\
    -\PIP_\lambda(a) \PIP_\lambda(U_k), &1 \leq k\leq a.
   \end{cases}
}
We use this expression to prove the following bound.

\begin{lemma}\label{prop1}
Let~$g_a$ be as defined at~\eq{23}. Then, for~$k\geq 0$, 
\bes{
  \abs{\Delta g_a(k)}&\leq \frac{1}{\lambda^{3/2}\sqrt{2e}} \bklr{\I\kle{{\ts k>a, k\geq \lambda}} + \I\kle{k<a,k<\lambda}}\\
	&\qquad\ +\left(\frac{\PIP_\lambda(a)}{a+1} +\frac{(\lambda-k)}{\lambda^2}\right)\I\kle{a<k<\lambda} \\
	&\qquad\ +\left(\frac{\PIP_\lambda(a)}{\lambda} +\frac{(k-\lambda)}{\lambda^2}\right)\I\kle{{\ts \lambda \leq k< a}}\\
	&\qquad\ + \frac{1}{\lambda} \I\kle{{\ts k=a}}\\
  & \leq \frac{1}{\lambda^{3/2}\sqrt{2e}} + \frac{\abs{\lambda-k}}{\lambda^2}%\I\bbkle{\abs{\lambda-k}<\abs{\lambda-a}}
+ \frac{1}{\lambda} \I[{\ts k=a}].
}
\end{lemma}

\begin{proof} The second bound follows from the first by noting that~$\PIP_\lambda(a)/(a+1)=\lambda^{-1} \PIP_\lambda(a+1)$;
here and below we use the bound~$\sup_{k\geq0}\PIP_\lambda(k)\sqrt{\lambda}\leq (2e)^{-1/2}$, from \cite[Proposition~A.2.7]{Barbour1992}.

The proof of the first bound consists of separate arguments in a number of cases, depending on the relative magnitudes of~$\la$, $k$ and~$a$.

\medskip
\noindent\textbf{Case 1:} Assume that~$k\geq a+1$. Then
\bes{
  \Delta g_a(k)&=\lambda^{-k-1} e^\lambda (k-1)! \PIP_\lambda(a) 
                \bklr{k \PIP_\lambda(U_{k+1}^c)-\lambda \PIP_\lambda(U_k^c)}\\
         &=\lambda^{-k-1} e^\lambda (k-1)! \PIP_\lambda(a) 
             \bbbklr{k\sum_{j=k+1}^\infty \frac{e^{-\lambda}  \lambda^j}{j!}-\sum_{j=k+1}^\infty  
                    \frac{e^{-\lambda}\lambda^j}{(j-1)!}}\\
         &=\lambda^{-k-1} e^\lambda (k-1)! \PIP_\lambda(a) \sum_{j=k+1}^\infty  \frac{e^{-\lambda}\lambda^j}{j!}(k-j).
}
We now use the fact that
\bes{
   I_k(\la) \Def \frac1{k!}\int_0^\la t^k e^{-t}\,dt   =\sum_{j=k+1}^\infty \frac{e^{-\lambda}  \lambda^j}{j!}&
      = -\frac{e^{-\la}\la^k}{k!} + I_{k-1}(\la)
}
to give
\besn{\label{24}
   \Delta g_a(k)&= \lambda^{-k-1} e^\lambda (k-1)! \PIP_\lambda(a)\{kI_k(\la) - \la I_{k-1}(\la)\}\\
   &=-\frac{\PIP_\lambda(a)}{k} \bklr{1+e^\lambda \lambda^{-k-1}(\lambda-k)\int_{0}^\lambda t^k e^{-t}\, dt} \\
   &=-\frac{\PIP_\lambda(a)}{k} \bbbklr{1+\frac{\lambda-k}{\lambda}\int_{0}^\lambda 
                        \left(\frac{t}{\lambda}\right)^k e^{\lambda-t}\, dt}.
}
\noindent\textbf{Subcase 1.1:} If~$k\geq \lambda$, then 
\besn{\label{25}
  0&\leq\frac{k-\la}{\lambda}\int_{0}^\lambda \left(\frac{t}{\lambda}\right)^k e^{\lambda-t}\, dt
            = \frac{k-\lambda}{\lambda}\int_{0}^\lambda \left(1+\frac{t-\lambda}{\lambda}\right)^k e^{\lambda-t}\, dt\\
    &\leq \frac{k-\lambda}{\lambda}\int_{0}^\lambda e^{(\lambda-t)(1-k/\lambda)}\, dt
          = 1-e^{\lambda-k}\leq1.
}
Therefore, in this case,
\bes{
\bbbabs{1-\frac{k-\lambda}{\lambda}\int_{0}^\lambda \left(\frac{t}{\lambda}\right)^k e^{\lambda-t} \,dt}\
           &= 1-\frac{k-\lambda}{\lambda}\int_{0}^\lambda \left(\frac{t}{\lambda}\right)^k e^{\lambda-t}\, dt \\
    &\leq 1-\frac{k-\lambda}{\lambda}\int_{0}^\lambda \left(\frac{t}{\lambda}\right)^k \,dt 
        = \frac{\lambda+1}{k+1},
}
and so
\be{
   \abs{\Delta g_a(k)}\leq \frac{\PIP_\lambda(a) (\lambda+1)}{k(k+1)}\leq \frac{1}{\sqrt{2e}} \lambda^{-3/2}.
}

\medskip
\noindent\textbf{Subcase 1.2:}
If~$a<k < \lambda$, then both summands in the expression for~$\Delta g_a(k)$ at the end of~\eq{24} are positive, and 
so, because~$I_k(\la) \le 1$, we have
\besn{\label{26}
\frac{\PIP_\lambda(a)}{k} \bbbklr{1+\frac{\lambda-k}{\lambda}\int_{0}^\lambda \left(\frac{t}{\lambda}\right)^k
               e^{\lambda-t} \, dt}
	&\leq \frac{\PIP_\lambda(a)}{k} \bbbklr{1+\frac{(\lambda-k)k!}{e^{-\lambda}\lambda^{k+1}}}\\	
	&=\frac{\PIP_\lambda(a)}{k} +\frac{(\lambda-k)}{\lambda^2} \frac{\PIP_\lambda(a)}{\PIP_\lambda(k-1)} \\
	&\leq \frac{\PIP_\lambda(a)}{a+1} +\frac{(\lambda-k)}{\lambda^2}, 
}
where, in the last inequality, we have used the unimodality of the Poisson distribution. 

\medskip
\noindent\textbf{Case 2:}
Assume that~$k\leq  a-1$. Then, if~$k\geq1$, following arguments similar to those above, we have
\bes{
   \Delta g_a(k)
	&=\lambda^{-k-1} e^\lambda (k-1)! \PIP_\lambda(a) \bklr{\lambda \PIP_\lambda(U_k)-k \PIP_\lambda(U_{k+1})}\\
	&=\lambda^{-k-1} e^\lambda (k-1)! \PIP_\lambda(a) \sum_{j=0}^{k-1} (j-k) \frac{e^{-\lambda} \lambda^j}{j!} \\
	&=\frac{\PIP_\lambda(a)}{k} \bklr{-1+e^\lambda \lambda^{-k-1}(\lambda-k)\int_{\lambda}^\infty t^k e^{-t} dt} \\
	&=\frac{\PIP_\lambda(a)}{k} \bbbklr{-1+\frac{\lambda-k}{\lambda}
              \int_\lambda^\infty \left(\frac{t}{\lambda}\right)^k e^{\lambda-t} \,dt}.
}
\noindent\textbf{Subcase 2.1:} If~$k<\lambda$,  we first take~$k=0$, where it is easy to see that 
\be{
  \abs{\Delta g_a(0)}=-g_a(1)=\lambda^{-1} \PIP_\lambda(a)\leq \frac{1}{\sqrt{2e}} \lambda^{-3/2}.
}
For~$1\leq k <\lambda$, an argument similar to~\eq{25} shows that 
\be{
  \abs{\Delta g_a(k)} = \frac{\PIP_\lambda(a)}{k} \bbbklr{1-\frac{\lambda-k}{\lambda}
                  \int_\lambda^\infty \left(\frac{t}{\lambda}\right)^k e^{\lambda-t} dt},
}
and by bounding~$t/\lambda >1$, we easily find that
\be{
\abs{\Delta g_a(k)}\leq \frac{\PIP_\lambda(a)}{\lambda}\leq \frac{1}{\sqrt{2e}} \lambda^{-3/2}.
}

\noindent\textbf{Subcase 2.2:} For~$\lambda \leq k <a$, following the same argument as in~\eq{26} yields that
\be{
 \abs{\Delta g_a(k)}\leq \frac{\PIP_\lambda(a)}{\lambda} +\frac{(k-\lambda)}{\lambda^2}
  \leq \frac{1}{\sqrt{2e}}\lambda^{-3/2} +\frac{(k-\lambda)}{\lambda^2}.
}

\medskip
\noindent\textbf{Case 3:} If~$k=a$, then we use the known bound~$\norm{\Delta g_a}_\infty<1/\lambda$ 
from \cite[Lemma~1.1.1]{Barbour1992}.
\end{proof}

For translated Poisson approximation, the bound in Lemma~\ref{prop1} easily translates into the following result. 

\begin{lemma}\label{lem7}
Let~$\mu,\sigma^2>0$, $s=\lfloor \mu-\sigma^2\rfloor$, $\gamma=\mu-\sigma^2-s$, and set~$\lambda=\sigma^2+\gamma$ and~$f_{a}(k)=g_{a}(k-s)$ for~$g_a$ the Poisson Stein solution defined at~\eq{23}. Then
\be{
  \abs{\Delta f_{a}(k)}\leq \frac{1}{\sigma^{3}\sqrt{2e}} + \frac{\abs{\mu-k}}{\sigma^{4}}  + \frac{\I\{{\ts k=a+s}\}}{\sigma^{2} }.
}
\end{lemma}

\subsection{Completing the proof of Theorem~\ref{thm1}}

In order to exploit Proposition~\ref{lem5}, we first need a manageable bound for the expectations
$\abs{\IE \sigma^2 \Delta f(W) - (W-\mu) f(W)}$ that appear there.  This is given in the following lemma.

\begin{lemma}\label{lem8}
Let~$(W,W',G,R)$ be an approximate Stein coupling with~$W$ and~$W'$ integer valued, ~$\IE W=\mu$ and~$\Var(W)=\sigma^2$.
Set~$D\Def W'-W$, and let~$\%F_1$ and~$\%F_2$ be sigma-algebras such that\/~$W$ is~$\%F_1$-measurable
 and such that\/~$(G,D)$ is~$\%F_2$-measurable.
Then
\ban{
&\abs{\IE \sigma^2 \Delta f(W) - (W-\mu) f(W)} \notag\\
\label{27}
	&\qquad\leq \babs{\IE\left[ \left(\IE [GD |\% F_1]-\IE GD\right) \Delta f(W)\right]}
	+\IE\abs{ R (W-\mu)}\IE\abs{ \Delta f(W)}+\IE\abs{R f(W)}
\\
	& \qquad\qquad + \IE\left[\frac{\abs{G D(D-1)}}{2}
       \min\left\{\norm{\Delta f} S_1(\law(W|\%F_2)), \norm{f} S_2(\law(W|\%F_2)) \right\}\right].\label{28}
}
\end{lemma}

\begin{proof}
Since~$(W, W', G,R)$ is an approximate Stein coupling,
\be{
   \IE(W-\mu)f(W) = \IE[G(f(W')-f(W))]-\IE[ Rf(W)],
}
and therefore
\ban{
   \babs{\IE [\sigma^2 \Delta f(W)& - (W-\mu) f(W)]} \notag\\
	&\leq \babs{\IE \left[ \sigma^2 \Delta f(W) - G(f(W')-f(W))\right]}+\IE \abs{R f(W)} \label{29}
}
With~$D=W'-W$, add and subtract~$G D \Delta f(W)$, and write~$\sigma^2 = \IE [GD] - \IE\{ R(W-\mu)\}$ 
using~\eq{7}, giving
\bes{
   &\IE\left[ \sigma^2 \Delta f(W) - G(f(W')-f(W))\right] \\
     &= \IE \left[(\IE [GD] - GD)\D f(W) - \IE\{ R(W-\mu)\} \D f(W) + GD\D f(W) - G(f(W')-f(W))\right].
}
Hence~\eq{29} can be bounded by
\ban{
  &\babs{\IE\left[ \left(\IE [GD |\% F_1]-\IE[ GD]\right) \Delta f(W)\right]}  +\IE\abs{ R(W-\mu)} \IE \abs{\Delta f(W)}+\IE\abs{R f(W)} \label{30}
\\
	&\quad +\babs{\IE \left[ G D \Delta f(W) - G  (f(W')-f(W))\right] } . \label{31}
}
It is easy to see that~\eq{30} is equal to~\eq{27}. For~\eq{31}, we observe that
\bes{
 & D \Delta f(W) - (f(W')-f(W)) \\
  &\quad=\I[D>0]\sum_{i=0}^{D-1} (\Delta f(W)- \Delta f(W+i) ) 
     - \I[D<0]\sum_{i=1}^{-D} (\Delta f(W) -\Delta f(W-i))\\
  &\quad= -\I[D>1]\sum_{i=1}^{D-1} \sum_{j=0}^{i-1} \Delta^2 f(W+j) 
      -\I[D<0]\sum_{i=1}^{-D} \sum_{j=0}^{i-1} \Delta^2 f(W-i+j).
}
Using this expression, conditioning on~$\% F_2$, and noting that for~$k\in\IZ$, 
\be{
  \abs{\IE[ \Delta^2 f(W+k)|\% F_2]}\
         \leq \min\left\{\norm{\Delta f} S_1(\law(W|\% F_2)),\,\norm{f} S_2(\law(W|\%F_2))\right\},
}
we find that~\eq{28} upper bounds~\eq{31}.
%The second assertion is essentially contained in the proof of \cite[Theorem~3.1]{Rollin2007}, but for clarity we provide a proof here.
%First note that~$(W, W', (W'-W)/(2a), R)$ is an approximate Stein coupling and 
%that by exchangeability,
%\bes{
%\IE [ G (f(W')-f(W))] & =\frac{1}{2a} \left(\IE \left[ \I[{\ts D=1}] \Delta f(W) \right] +\IE \left[\I[{\ts D=-1}] (f(W)-f(W')  \right]\right) \\
%		&=\frac{1}{a} \IE \left[ \I[{\ts D=1}] \Delta f(W) \right],
%} 
%which, from~\eq{7}, implies~$\IP(D=1)/a= \sigma^2+\IE[R (W-\mu)]$.
%These properties and ~\eq{29} imply
%\bes{
% &\babs{\IE [\sigma^2 \Delta f(W) - (W-\mu) f(W)]}  \\
% 	&\qquad\leq \babs{\IE \left[ \left(\sigma^2-a^{-1} \IP({\ts D=1}|\% F_1)\right)\Delta f(W)  \right]}+\IE \abs{R f(W)} \\
%	&\qquad \leq \frac{1}{a}\babs{\IE \left[ \left(\IP({\ts D=1})- \IP({\ts D=1}|\% F_1)\right)\Delta f(W)  \right]}+\IE\abs{ R (W-\mu)}\IE\abs{ \Delta f(W)}+\IE \abs{R f(W)},
%}
%as desired.
\end{proof}

\begin{proof}[Proof of Theorem~\ref{thm1}]
The total variation bound \eq{8} is a consequence of 
Proposition~\ref{lem5} and Lemma~\ref{lem8}, together with the bounds from Lemma~\ref{lem6}; c.f. \cite[Theorem~3.1]{Rollin2007} and \cite[Theorem~1.3]{Fang2014}).
To prove \eq{9}, we can argue similarly, but taking~$f = f_a$ and using
Lemma~\ref{lem7} to bound~$\D f_a(\cdot)$.  This yields \eq{9}: the terms in \eq{9} (except for the last one) bound~\eq{27},
and the last term in \eq{9} bounds~\eq{28} and the extra term~$2\s^{-2}$ in Proposition~\ref{lem5}.
\end{proof}

\section{Proofs of applications}\label{sec4}

In this section we prove the application results given in Section~\ref{sec2}.

\subsection{Hoeffding combinatorial local central limit theorem}
Recall the definitions of~$a$, $W$, $\mu$ and~$\sigma^2$ from Section~\ref{sec2} and also Assumptions~A1 and~A2.

The first step is to find a Stein coupling for~$W$. We choose among those discussed in \cite[Section~4.1]{Chen2010}. 
The most direct procedure is to let~$I,J$ be i.i.d.\ uniform on~$\{1,\ldots,n\}$, and to define
$W'\Def W-a_{I \ppi_I}-a_{J \ppi_J} + a_{I \ppi_J} + a_{J \ppi_I}$; then~$(W,W')$ is an exchangeable pair satisfying 
$\IE [W'-W|W]=-\frac{2}{n-1} (W-\mu)$, giving an exact Stein coupling with~$G\Def (n-1) (W'-W)/4$.  However, we use
another exact Stein coupling, that yields a simpler form for~$T$ in~\eq{10}; 
we take
\[
     W' \Def W-a_{I \ppi_I} - a_{J \ppi_J},\ {\ts I\not=J};\quad 
     W' \Def W-a_{I \ppi_I},\ {\ts I=J};\quad G \Def n(a_{I \ppi_J}-a_{I \ppi_I}).
\]
Now,
$\IE[GD|\ppi]= 2n^{-1}\mu(W-\mu) + \sum_{l=1}^4 T_l + n^{-1}\mu^2$, where
\be{
  T_1 \Def \sum_{i} a_{i \ppi_i}^2, \quad T_2 \Def -\frac{1}{n}\sum_{i} a_{i \ppi_i} a_{i +},\quad
  T_3 \Def -\frac{1}{n}\sum a_{+ \ppi_j} a_{j \ppi_j},\quad T_4 \Def \frac1n(W-\mu)^2,
}
and so
\bes{
   \babs{\IE[GD|\ppi]-\IE[GD]} &\leq 2A_1|W-\mu| + T, \quad \mbox{with}\quad T \Def \sum_{l=1}^4 |T_l - \IE T_l|,
}
satisfying condition~\eq{10} with~$k=1$ and~$\k = 2A_1$.
We now consider the remaining conditions to be satisfied in Corollary~\ref{17}\emph{(i),(ii)}.
Since~$R=0$, for~$\emph{(i)}$, we need to bound~$\IE T$ and~$\U$. Note that since
$D \Def  W'-W$ satisfies~$|D| \le 2A_1$, and also that~$\s^{-2}|G| \le 2/(\a_0^2 A_1)$, we have
$\U\leq C \, \sigma^2 \IE[ S_2(\law(W)|\% F_2)]$, 
where we define~$\% F_2$ be the sigma-algebra generated by~$(I,J, \rho_I,\rho_J)$.

\smallskip
\noindent\textbf{The smoothing term.} 
We begin with the smoothing coefficient
%$S_1(\law(W))$ and~
$\IE[ S_2(\law(W)|\% F_2)]$; recall Assumption~A2. 
 \begin{lemma}\label{32}
 Under Assumptions A1 and~A2, we have
\be{
\IE [S_2(\law(W)|\% F_2)] = \bigo(n^{-1}).
}
\end{lemma}

\begin{proof}
Condition on~$I,J,\rho_I,\rho_J$, let~$\ppi^\circ$ be a uniformly chosen permutation 
given~$\rho^\circ_k=\rho_k$ for~$k=I,J$, and define~$W^\circ=\sum_{i=1}^n a_{i\rho^\circ_i}$ so
that~$\law(W^\circ)=\law(W|\% F_2)$. Now, 
for each~$1\le l\le n_1$ such that neither of~$i_{l1},i_{l2}$ are equal to~$I$ or~$J$, 
independently 
and with probability~$1/2$, multiply~$\ppi^\circ$ by the transposition~$(\ppi^\circ_{i_{l1}},\ppi^\circ_{i_{l2}})$, so that, if 
the multiplication takes place, then~$i_{l2}\mapsto \ppi^\circ_{i_{l1}}$ and~$i_{l1}\mapsto\ppi^\circ_{i_{l2}}$. This process forms a 
new permutation~$\tilde \ppi$, which is still uniformly distributed given~$\tilde \rho_k=\rho_k$ for~$k=I,J$, so that~$\tW\Def \sum_{i=1}^n a_{i \tilde\ppi_i}$ 
has the same distribution as~$W^\circ$. Moreover, writing
\[
   C_l(\ppi^\circ) \Def 
       \bklr{a_{i_{l1},\ppi^\circ_{i_{l2}}}+a_{i_{l2},\ppi^\circ_{i_{l1}}}-a_{i_{l1},\ppi^\circ_{i_{l1}}}-a_{i_{l2},\ppi^\circ_{i_{l2}}}},
\]
we have for~$E_{IJ}\Def \left\{l : 1\leq l \leq n_1, \{i_{l1},i_{l2}\}\cap\{I,J\}=\emptyset\right\}$,
\be{
  \tW = W^\circ +\sum_{l\in E_{IJ}} B_l C_l(\ppi^\circ),
}
where~$B_1,\ldots, B_{n_1}$ are i.i.d.\ Bernoulli~$\Be(1/2)$ random variables, independent of~$\rho$.
Defining 
\be{
N(\ppi^\circ) \Def \abs{\{l\colon \{{\ts i_{l1},i_{l2}\}\cap\{I,J\}=\emptyset, \abs{C_l(\ppi^\circ)}=1}\}},
}
it thus follows immediately that, on the
event~$\{N(\ppi^\circ) \ge k_n\}$, we have
\be{
              S_2(\law(\tW\giv\rho^\circ)) \leq S_2(\Bi(k_n,1/2)) \leq 10k_n^{-1},
}
where the last inequality is 
\cite[Proposition~3.8]{Rollin2015}.

Taking expectations, this in turn implies that
\be{
  S_2(\law(W|\% F_2)) = S_2(\law(W^\circ)) = S_2(\law(\tW)) \leq 8 k_n^{-1} + 4\IP(N(\ppi^\circ) < k_n). 
}
Defining~$k_n \Def  \lfloor \tfrac12 \IE N(\ppi^\circ)\rfloor$, we  show that, for suitable~$\b_1 > 0$, $\b_2 < \infty$, we have~$k_n \ge \b_1 n$ and~$\IP(N(\ppi^\circ) < k_n) \le \b_2 n^{-1}$ for large~$n$, thus completing the proof of the lemma.

Note that ~$\rho^\circ$ is a uniformly chosen map from
~$\{1,\ldots,n\}\setminus\{I,J\}$ to~$\{1,\ldots,n\}\setminus\{\rho_I, \rho_J\}$. Thus, from our assumption on the matrix~$a$,
\bes{
  \IE N(\ppi^\circ)&\geq -2+  \sum_{l\in E_{IJ}}   \IP\bklr{\ts\abs{C_i(\ppi^\circ)}=1}\\
      & = -2+ \sum_{l\in E_{IJ}}   \IP\bklr{\{\ppi^\circ_{i_{l1}},\ppi^\circ_{i_{l2}}\} \in \cJ(i_{l1},i_{l2})} \\
     &\geq -2+ \frac{(n_1-2)(n_2-2n)}{n(n-1)} \geq n \left(\left(\a_1-\frac{2}{n}\right)\left(\a_2-\frac{2}{n}\right)-\frac{2}{n}\right),
}
and so~$k_n\ge \b_1 n$ for~$n$ large, with~$\b_1 \Def  \a_1\a_2/4$.  Then, by Chebyshev's inequality, we have the upper bound
\ben{\label{33}
  \IP({\ts N(\ppi^\circ)\leq k_n}) %= \IP(N(\ppi)-\IE N(\ppi) \leq k_n-\IE N(\ppi))
      \leq  \frac{\Var(N(\ppi^\circ))}{(\IE N(\ppi^\circ) - k_n)^2} = \frac{4\Var(N(\ppi^\circ))}{(\IE N(\ppi^\circ))^2}.
}
It is now enough to show that~$\Var(N(\ppi^\circ))\leq Cn$, for some~$C < \infty$.

To bound~$\Var(N(\ppi^\circ))$, we need to compute the covariance of~$X_l$ and~$X_{k}$ for~$l\not= k$, both in~$E_{IJ}$, where
\[
     X_r \Def \I\bklr{\{\ppi^\circ_{i_{r1}},\ppi^\circ_{i_{r2}}\} \in \cJ(i_{r1},i_{r2})}.
\]
Now, given~$X_l=1$ and~$(\ppi^\circ_{i_{l1}},\ppi^\circ_{i_{l2}}) = (j_1,j_2)$, the conditional probability that~$X_k=1$ can be no
larger that~$\IP(X_k=1) (n-2)(n-3)/(n-4)(n-5)$, since pairs that are excluded by having~$(j_1,j_2)$ as images of
$i_{l1}$ and~$i_{l2}$ under~$\ppi$ would only reduce the conditional probability, and the probability of an
accessible pair being attained is increased by the factor~$(n-2)(n-3)/(n-4)(n-5)$.  Hence
\ba{
     \Cov(X_l,X_k) &\leq \IP({\ts X_l=1}) \IP({\ts X_k=1})\,\left\{\frac{(n-2)(n-3)}{(n-4)(n-5)} - 1 \right\} \\
                   &\leq \IP({\ts X_l=1}) \IP({\ts X_k=1}) \,\frac{2(2n-7)}{(n-4)(n-5)},
}
so that, for~$n\ge28$,
\bes{
\Var(N(\ppi^\circ))&= \sum_{l\in E_{IJ}}  \Var(X_l)+ \sum_{l\not=k; l,k\in E_{IJ} } \Cov(X_l,X_k) \\
	&\leq \IE N(\ppi^\circ) + 5n^{-1}\{\IE N(\ppi^\circ)\}^2 = \bigo(n).
}
This proves the lemma.
\end{proof}

As a consequence of Lemma~\ref{32} and the remarks preceding it, we have~$\U = \bigo(1)$, and we now show
$\IE T = \bigo(\s)$, after which, Corollary~\ref{17}\emph{(i)} implies
\ben{\label{34}
   \dtv\left(\law(W),\TP\bklr{\IE W, \Var(W)}\right) = \bigo\left(1/\s\right).
}
  Observing first that
\[
     \IE[T\I[{\ts W=k}]] = \sum_{l=1}^4 \IE[|T_l - \IE T_l|\I[{\ts W=k}]],
\]
we apply Lemma~\ref{lem4} to the first three terms; for the fourth, we immediately have
\[
   \IE\bigl[|T_4 - \IE T_4|\I({\ts W = k})\bigr] \leq 2\IE T_4 = n^{-1}\s^2 = O(1),
\]
so that this element of~$T$ gives a contribution to the error bound of order~$\bigo(\s^{-2})$.

The remaining elements each
have the form~$\babs{\sum_i \tilde a_{i \ppi_i}- \IE \sum_i \tilde a_{i \ppi_i}}$,
for appropriate choices of~$\tilde a_{ij} \leq A_1^2$ (take~$\tilde a_{ij}$ to be~$a_{i j}^2$, $a_{i j} a_{i +}/n$ and 
$a_{i j} a_{ + j}/n$, respectively). It thus follows from~\eq{19} that
\[
     \IE|T_l - \IE T_l| \leq \sqrt{\Var T_l} = \bigo(\sqrt n),\quad\ts 1\le l\le 3,
\]
and hence that~$\IE T = \bigo(\s)$, as desired.
 
For the local approximation, 
Lemma~\ref{32} and~\eq{34} imply that Condition~(1) of 
of Corollary~\ref{17}\emph{(ii)} is satisfied with~$\a=1$.  For Condition~(2), under Assumption~A1,
\cite[Proposition~1.2]{Chatterjee2007} implies that, for any~$j\geq 0$, 
\be{
   \s^{-2j}\IE (W-\mu)^{2j} \leq (2j-1)^{2j} A_1^{2j} \s^{-2j} n^{j} = \bigo(1). 
}
Finally, we (essentially) use Condition~(3b) for the~$T$ term and treat~$T_1, T_2$ and~$T_3$ using concentration bounds.
Under Assumption~A1, 
\cite[Proposition~1.1]{Chatterjee2007} and \cite[Theorem~3.1]{Goldstein2014} 
imply, for~$l=1,2,3$, that
\bes{
%  \IP\bbklr{\abs{T_i}\geq t\s}
%     &\leq 2 \exp\left\{ -\frac{ 3t^2}{2A_2^2\bklr{6(\nu_i/\sigma^2) +  t/\sigma}}\right\}, \\
  \IP\bbklr{\ts(A_1\s)^{-1}\abs{T_l - \IE T_l}\geq t}
     &\leq 2 \exp\left\{ -\frac{ t^2}{2 (v_l/(A_1\sigma)^2) + 16(A_1/\sigma)t}\right\},
}
where~$v_l=\Var(T_l)$, and ~$v_l/(A_1\s)^2 \le C_l$ for some suitable~$C_l$.  We can then 
apply Lemma~\ref{lem4}, taking
\[
     \eps_l(t) = t\bF_l(t) + \int_t^\infty \bF_l(v)\,dv,
\]
where 
\[
     \bF_l(t) \Def 2 \exp\left\{ -\frac{ t^2}{2 C_l + (16/\a_0)tn^{-1/2}}\right\},
\]
and the choice~$t = C'_l\sqrt{\log\s}$, for~$C'_l$ suitably large but fixed, gives 
\[
     \IE[|T_l - \IE T_l|\I[{\ts W=k}]] \leq  \s\eps_l(t) + t = \bigo(\sqrt{\log \s}).
\]
Hence, from Corollary~\ref{17},  under Assumptions A1 and~A2, we have
\[
   \dloc\left(\law(W),\TP\bklr{\IE W, \Var(W)}\right) = \bigo\left(\frac{\sqrt{\log(\s)}}{\s^2}\right).
\]

\subsection{Number of isolated vertices in an \ER\ random graph}

Let~$\cG\Def \cG(n,p)$ be an \ER\ random graph on~$n$ vertices~$v_1,\ldots,v_n$, and let~$W$ be the number 
of isolated vertices in~$\cG$. Let~$W^s$ have the size-biased distribution of~$W$. Then~$(W,W',G)=(W,W^s, \IE W)$ 
is a Stein coupling. To couple~$(W,W^s)$, construct~$W^s$ from~$\cG$ by choosing a vertex at random 
and erasing all edges (if any) connected to the vertex.
Recall from the introduction that
we consider the regime~$p\asymp \lambda/n$ for some~$\lambda>0$, in which 
case~$\mu \sim ne^{-\la}$ and~$\sigma^2 \sim ne^{-\la}\{1 + (\la-1)e^{-\la}\}$ are of strict order~$n$.
%We note that
%\be{
%   \mu \Def \IE W = n (1-p)^{n-1};\quad  \sigma^2 \Def \Var(W) = n(1-p)^{n-1}\bkle{1+(np-1)(1-p)^{n-2}},
%}
%and we consider the regime~$p\asymp \lambda/n$ for some~$\lambda>0$, in which 
%case~$\mu \sim ne^{-\la}$ and~$\sigma^2 \sim ne^{-\la}\{1 + (\la-1)e^{-\la}\}$ are of strict order~$n$.

Let~$E_i$ be the event that vertex~$v_i$ is not isolated in~$\cG$, let~$W_1(v)$ be the number of degree-$1$ vertices
connected to vertex~$v$ of~$\cG$, and let~$W_1$ be the number of degree-$1$ vertices in~$\cG$. To check~\eq{10}, 
we observe that

\bes{
  \babs{\IE[GD| \cG]-\sigma^2}&= \bbabs{\frac{\mu}{n}\sum_{i=1}^n \bklr{W_1(v_i)+\I\{E_i\}}-\sigma^2} \\	
	&= \bbabs{\frac{\mu}{n}\bklr{W_1+(n-W)}-\sigma^2}.
}
Since~$\IE[GD] = \s^2$, which is \eq{7} for a Stein coupling (thus with~$R=0$), it follows that	
\bes{
  \babs{\IE[GD| \cG]-\sigma^2} &= \bbabs{\frac{\mu}{n}\bbklr{(W_1 - \IE W_1) - (W - \IE W)}}\\
	&\leq   (1-p)^{n-1}\{\abs{ W-\mu} + \abs{W_1-\IE W_1}\} \\
	&\leq \abs{ W-\mu}+   \abs{W_1-\IE W_1}, 
}
which is~\eq{10} with~$k=1$, $\k=1$ and~$T=\abs{W_1-\IE W_1}$.

To apply Corollary~\ref{17}, we need to show that~$\s^{-j}\IE|W-\mu|^j = \bigo(1)$ for suitable values
of~$j$.  To do so,
and also to show that the distribution of~$T$ is concentrated,
we take~$d=0$ and~$d=1$ in the following theorem of  \cite{Bartroff2015} 
(see also \cite{Arratia2015}).

\begin{theorem}\label{thm2}
For any integer~$d\geq0$, let~$W_d$ be the number of degree~$d$ vertices in an \ER\ random graph~$\cG$ with parameters 
$n$ and~$p$. Then, for any~$t>0$,
\be{
  \IP\bklr{\abs{W_d-\IE W_d}> t}\leq 2 \exp\left\{-\frac{t^2}{4(n-\IE W_d) +(4/3)t}\right\}.
}
\end{theorem}

So, for any value of~$d$, we have
\bes{
   \IP\bklr{\ts\s^{-1}\abs{W_d-\IE W_d} \geq t} 
     &\leq 2 \exp\left\{-\frac{t^2}{4\frac{n}{\sigma^2} +\frac{4}{3\sigma} t}\right\} \\
     &\leq \h_n(t) \Def  2  \exp\left\{-\frac{t^2}{4\ess +4t\sqrt{\ess/n}/3}\right\},
}
where~$\ess$ is an upper bound for~$n/\s^2$.  It follows easily, taking~$d=0$, that 
$\s^{-k}\IE|W-\mu|^k = \bigo(1)$ for all~$k\ge1$, and then, taking~$d=1$ and
\[
     \eps_n(t) \Def t\h_n(t) + \int_t^\infty \h_n(v)\,dv,
\]
that~$\int_1^\infty \eps_n(t)\,dt \le \int_1^\infty \eps_1(t)\,dt < \infty$ for all~$n\ge1$, so that 
$\IE T^2 = \bigo(\s^2)$. Since also~$R=0$ almost surely, and since \cite[Lemma~4.7]{Rollin2015} shows that
\be{
   S_2\bklr{\law(W)} = \bigo(\sigma^{-2}),
}
then once we show~$\U = \bigo(1)$, all the hypotheses and conditions of Corollary~\ref{17}\emph{(i),(ii)} except for~(3) of \emph{(ii)} are satisfied, with~$\a = 1$.
But a variation of Condition~(3b) is satisfied: for~$t = t_n = c\sqrt{\log\s}$, it is easy to check that
$\eps_n(t_n) = \bigo(n^{-1/2})$ if~$c>0$ is chosen fixed but large enough, so that, from
Lemma~\ref{lem4}, $\IE[T\I[W=k]] = \bigo(\sqrt{\log\s})$, and thus 
the contribution from~$T$ is at most of order~$\bigo(\s^{-2}\sqrt{\log\s})$.

All that is left is to to show that~$\U = \bigo(1)$, for which we follow \cite{Fang2014}. Let~$I$, uniformly distributed on~$\{1,\ldots, n\}$, 
be the index
of the vertex of~$\cG$ chosen to be isolated in constructing~$W^s$, and for~$k=0,1,2$, let~$\cN\sp i_k$ be the 
set of vertices at distance~$k$ from vertex~$v_i$ in~$\cG$.  Then let~$\% F_2$ be the sigma algebra generated by 
$\bigl(I, \cN\sp I_1, \cN\sp I_2\bigr)$ and the presence or absence of all edges that have one or more vertices
 in~$\{I\}\cup\cN\sp I_1$. %with a vertex in~$\cN\sp I\Def \cN\sp I_1\cup \cN\sp I_2$.
Clearly~$W^s-W$ is~$\% F_2$-measurable.
To bound~$\U \Def  \IE \bkle{\abs{G D(D-1)}S_2(W|\%F_2)}$, consider the expectation on the 
event~$\{\abs{\cN_1\sp I}> \sqrt{n}\}$ and on its complement.

First, note that~$\abs{D}=\abs{W^s-W}\leq 1 + \abs{\cN_1\sp I}$, so that
\bes{
  \IE \bbkle{\abs{G D(D-1)}S_2(W|\%F_2)\I\bklg{\abs{\cN_1\sp I}> \sqrt{n}}}
    &\leq C n \IE \bbkle{(1+\babs{\cN_1\sp I}^2)\I\bklg{\abs{\cN_1\sp I}> \sqrt{n}}}.
%    &\leq C n \bbbklr{\IE\bbkle{\abs{\cN\sp I}^2} \IP\bbklr{\abs{\cN\sp I}> \sqrt{n}}}^{1/2}.
}
Since~$\abs{\cN_1\sp I} \sim \Bi(n-1,p)$, and~$p \sim \la/n$,
% is dominated by~$\sum_{i=1}^X X_i$, where~$X, X_1,\ldots$ are i.i.d.~$\Bi(n,p)$,
\be{
    \IE\bbkle{ \abs{\cN_1\sp I}^k} \leq C_k\quad\text{for all~$n\ge1$},
}
for suitable constants~$C_k$, so that
$$
     \IE \bbkle{(1+\babs{\cN_1\sp I}^2)\I\bklg{\abs{\cN_1\sp I}> \sqrt{n}}} = \bigo(n^{-k/2})
$$
for all integers~$k\geq1$.

For the complementary event, we show that, for some universal constant~$C$, 
\ben{\label{35}
    S_2(W|\%F_2)\I\bklg{{\ts\abs{\cN_1\sp I} \leq \sqrt{n}}} \leq C \sigma^{-2}, \text{a.s.}
}
If this is the case, then 
\be{
  \IE \bbkle{\abs{G D(D-1)}S_2(W|\%F_2)\I\bklg{{\ts\abs{\cN_1\sp I} \leq  \sqrt{n}}}} 
     \leq C n \sigma^{-2} \IE\bbkle{(1+ \abs{\cN_1\sp I}^2)} = \bigo(1),
}
as desired. For~\eq{35}, the basic idea is that there still remain
almost~$\binom{n}2$ edges to be independently assigned, and the 
methods leading to \cite[Lemma~4.7(i)]{Rollin2015} can be applied to give the required order.

From now on, we have~$\abs{\cN_1\sp I}\leq  \sqrt{n}$.
Given~$\% F_2$, define a new random graph~$\tcG$ on~$n$ vertices labeled~$\{v_1,\ldots,v_n\}$ such 
that all edges with an endpoint in~$V(I) \Def  \bklg{v_i\colon i\in\{I\}\cup\cN\sp I_1}$  are determined by~$\% F_2$, and  
the remaining edges, those in~$E(I) \Def  \{\{i,j\}\colon i,j \notin V(I)\}$,  
are assigned using i.i.d.~$\Be(p)$ variables; we let~$\tcG(I)$ denote the graph~$\tcG$ restricted to~$E(I)$. 
Note that the number of edges in~$E(I)$ is 
\be{
  \binom{n-\abs{\cN\sp I_1}-1}{2} \sim \frac12\,n^2,
}
because~$\abs{\cN\sp I_1}\leq  \sqrt{n}$.
 Let~$\tcG'$ be the graph obtained by choosing at random  one of the 
 edges of~$E(I)$  and resampling it, and let~$\tcG''$ be the graph  
obtained from the same operation applied to~$\tcG'$.  Let~$\tW,\tW', \tW''$ be the number of 
isolated vertices in~$\tcG, \tcG', \tcG''$. Then~$\law(\tW)=\law\bklr{W\vert \% F_2}$, 
and~$(\tW, \tW', \tW'')$ are three successive states of a reversible Markov chain. 
Thus \cite[Theorem~3.7]{Rollin2015} implies that 
\bes{
   &S_2\bklr{\law(W|\%F_2)} \\
    &\quad\leq \frac{1}{\IP(\tW'=\tW+1)^2}\bbbkle{ 2\Var\bbklr{\IP\bklr{\ts\tW'=\tW+1|\tilde \cG}} 
           +2\Var\bbklr{\IP\bklr{\ts\tW'=\tW-1|\tilde \cG}} \\
   &\hspace{4.25cm} +\IE \bbabs{
           \IP\bklr{\ts\tW''=\tW'+1, \tW'=\tW+1\vert \tilde \cG}-\IP\bklr{\tW'=\tW+1\vert \tilde \cG}^2} \\
   &\hspace{4.5cm}+\IE \bbabs{
          \IP\bklr{\ts\tW''=\tW'-1, \tW'=\tW-1\vert \tilde \cG}-\IP\bklr{\ts\tW'=\tW-1\vert \tilde \cG}^2}
                 }.
} 
Bounds on the first two terms 
are given by \cite[Inequalities~(2.23)-(2.25)]{Fang2014}, which yield 
\be{
   \IP({\ts\tW'=\tW+1})\geq Cn^{-1} \,\, \mbox{ and } \,
            \Var\bbklr{\IP\bklr{{\ts\tW'=\tW\pm 1}\given\tilde \cG}} \leq Cn^{-3}.
}

For the last two terms of the bound, let~$\cV_1\sp I$
be the set of vertices having degree one in both of~$\tcG$ and~$\tcG(I)$, and let~$\hcV_1\sp I$ be the subset 
of these vertices that are connected to a vertex having degree two in both of~$\tcG$ and~$\tcG(I)$;
write~$\tW_1\sp I \Def  |\cV_1\sp I|$ and~$\hW_1\sp I \Def  |\hcV_1\sp I|$. 
Let~$\cE_2\sp I$ be the set of edges that are  isolated in both~$\tcG$ and~$\tcG(I)$, 
and let~$\cE_3\sp I$ be the set of pairs of connected edges that are isolated in both~$\tcG$ and~$\tcG(I)$;
denote their numbers by~$E_2\sp I$ and~$E_3\sp I$ respectively. 
  Note that no vertices of~$N_1\sp I$ are isolated, but that~$v_I$ may be isolated (and then~$N_1\sp I$
is empty); note also that the endpoints of elements of~$\cE_3\sp I$ belong to~$\hcV_1\sp I$.

Now the only way to increase the number of isolated vertices in going from~$\tcG$ to~$\tcG'$ is to choose 
a non-isolated edge connected to a degree one vertex, and then remove it; 
however, the number of isolated vertices
increases by~$2$ if the edge removed belongs to~$\cE_2\sp I$.  Hence, writing~$\hn(I) \Def  n-\abs{\cN_1\sp I}-1$,
we have
\be{
   \IP\bklr{{\ts \tW'=\tW+1 \giv \tilde \cG}} = \frac{(\tW_1\sp I-2\tE_2\sp I)}{\binom{\hn(I)}{2}}\,(1-p).
}
Considering the different ways of increasing the number of isolated vertices by exactly one in consecutive steps
is more complicated;  isolating a vertex in~$\hcV_1\sp I$ leaves the number of vertices
of degree~$1$ unchanged, so that~$(\tW_1\sp I)' = \tW_1\sp I$,  but if the vertex belonged to an element
of~$\cE_3\sp I$, then~$(\tE_2\sp I)' = \tE_2\sp I + 1$. Hence
\bes{
   &\IP\bklr{{\ts \tW''=\tW'+1, \tW'=\tW+1\given \tcG}} = 
          \frac{\bklr{\tW_1\sp I-2\tilde E_2\sp I-\tilde V\sp I_1}\bklr{\tW_1\sp I - 1 -2\tilde E_2\sp I}}
                    {\binom{\hn(I)}{2}^2}(1-p)^2 \\
   &\qquad+\frac{\bklr{\tilde V\sp I_1-2\tilde E_3\sp I}\bklr{\tW\sp I_1-2\tilde E_2\sp I}}
                    {\binom{\hn(I)}{2}^2}(1-p)^2
          +\frac{2\tilde E_3\sp I\bklr{\tW\sp I_1-2\tilde E_2\sp I-2}}
                    {\binom{\hn(I)}{2}^2}(1-p)^2,
}
so that
\bes{
   &\IE \bbabs{\IP\bklr{{\ts\tW''=\tW'+1, \tW'=\tW+1\given \tcG}}-\IP\bklr{\tW'=\tW+1\given \tcG}^2} \\
    &\qquad\leq \frac{\IE \abs{\tW_1\sp I-2\tE_2\sp I}+\IE \hW_1\sp I+4\IE \tE_3 \sp I}
               {\binom{\hn(I)}{2}^2}\,(1-p)^2 = \bigo(n^{-3}),
}
since~$\abs{\tW_1\sp I-2\tE_2\sp I}\leq \tW_1\sp I\leq n$, $\hW_1\sp I\leq n$, $\tE_3\sp I\leq n$, and 
$\abs{\cN_1\sp I}\leq  \sqrt{n}$. 
Similarly, but more easily, we  have
\be{
    \IP\bklr{{\ts \tW'=\tW-1 \giv \tcG}} = \frac{(\tW-\I[\deg(v_I)=0])(\hn(I)-\tW+\I[\deg(v_I)=0])}
                {\binom{\hn(I)}{2}}\,p,
}
and 
\be{
     \IP\bklr{{\ts\tW''=\tW'-1, \tW'=\tW-1 \given \tcG}} = 
       \frac{4\binom{\tW-\I[\deg(v_I)=0]}{2}\binom{\hn(I)-\tW+\I[\deg(v_I)=0])+1}{2}}{\binom{\hn(I)}{2}^2}\,p^2,
}
so that
\bes{
&\IE \bbabs{{\ts\IP\bklr{\tW''=\tW'-1, \tW'=\tW-1\given \tcG}}-\IP\bklr{{\ts \tW'=\tW-1\vert \tcG}}^2} \\
    &= \frac{ (\tW-\I[\deg(v_I)=0])(\hn(I)-\tW+\I[\deg(v_I)=0]) (\hn(I)-2\tW+2\I[\deg(v_I)=0]+1)p^2}
           {\binom{\hn(I)}{2}^2},
}
which is again~$\bigo(n^{-3})$, since~$0\leq \tW \leq n$.
Therefore~$S_2\bklr{\law(W|\%F_2)}=\bigo(n^{-1})$ almost surely, as desired.

\subsection{Curie--Weiss}

Recall from Section~\ref{sec2} the definition of the Curie--Weiss distribution, the magnetization~$W$, and associated 
discussion.
Assume that either~$h>0$ and~$\beta>0$, or that~$h=0$ and~$0<\beta<1$.
Define the exchangeable pair~$(W,W')$ as follows. Let~$I$ be uniform on~$\{1,\ldots,n\}$. Given~$I=i$ and~$S=s$, let 
\be{
    \IP({\ts S_i'=x})=\IP\bklr{{\ts S_i=x\given(S_j)_{j\not=i}=(s_j)_{j\not=i}}}
}
for~$x=\pm1$. Defining~$W'\Def W-S_I+S_I'$, we note that~$(W,W')$ are two consecutive states of a stationary Gibbs sampler, and so form an 
exchangeable pair. 
Note that~$W$ actually sits on a lattice of span~$2$ (even or odd numbers, depending on~$n$), so that our 
eventual conclusion concerns ~$\tW \Def  (W+\tfrac12\{1-(-1)^n\})/2$.

We next want to establish an approximate linear regression, so as to determine an approximate Stein coupling.
From \cite[Page~315]{Chatterjee2007} (see also \cite[(7.10)]{Ross2011}),  for~$\mu_n\Def \IE W$, we have
\ban{
   &\IE[W'-W\given S] = -\frac{1}{n} W +\frac{1}{n} \sum_{i=1}^n \tanh\left(\frac{\beta}{n} (W-S_i) + h\right)
                  \nonumber \\
  &\qquad = -\frac{1}{n} (W-\mu_n)+ \frac{1}{n}\sum_{i=1}^n \left(\tanh\left(\frac{\beta}{n} (W-S_i)+ h\right)
                   -\tanh\left(\frac{\beta}{n} W+ h\right)\right) \nonumber\\
  &\qquad\qquad +\tanh\left(\frac{\beta}{n} W+ h\right)-\tanh\left(\beta m_h+ h\right) +(m_h-\mu_n/n).&\label{36}. 
}
Now since, $0 \le \frac d{dx}\tanh(x)=1-\tanh^2(x) \le 1$ and~$|\frac{d^2}{dx^2}\tanh(x)| \le 1$, by Taylor expansion we have 
\bg{
   |\tanh(\b w + h) - \tanh(\b m + h) - \b(w-m)(1-\tanh^2(\b m+h))| \leq C|w-m|^2;\\
      |\tanh(\b(w-s)+h) - \tanh(\b w + h)| \leq \b s,
}
from which it follows that 
\ben{
   \Bigl|\tanh\left(\frac{\beta}{n} W+ h\right)-\tanh\left(\beta m_h+ h\right)  - n^{-1}\b(W - nm_h)(1 - m_h^2)\Bigr| 
             \leq C\left(\frac Wn-m_h\right)^2, \label{37}
}
and
\ben{
    \bbbabs{  \frac{1}{n}\sum_{i=1}^n \left(\tanh\left(\frac{\beta}{n} (W-S_i)+ h\right)-\tanh\left(\frac{\beta}{n} W+ h\right)\right)}
              \leq \frac{\b} n. \label{38}
}
This gives an approximate linear regression~$\IE[W'-W \giv W]  = -a(W-\mu_n) + aR$, with
%~$a=(1-\beta (1-m_h^2))/n$ and 
\ben{\label{39}
   a \Def n^{-1}(1-\beta (1-m_h^2));\quad |R| \leq
     R' \Def \frac{\beta}{1-\beta (1-m_h^2)} + \abs{\mu_n-n m_h}+\frac{C n \abs{W/n-m_h}^2}{1-\beta (1-m_h^2)},
}
and the approximate Stein coupling is completed by taking~$G\Def (W'-W)/(2a)$. 
Below we work on~$(W,W', G,R)$, but note that 
all results easily transfer to~$(\tW,\tW',\tG,\tR)$ where~$\tW$ is as above,
$\tW'$ is defined in the obvious way, $\tG=G/2$ and~$\tR'=R'/2$. In this case, $(\tW,\tW')$ satisfy~\eq{5}
and~$\abs{\tW'-\tW}\leq 1$, so we apply our approximation
framework,
 using Remark~\ref{rem1}.

The first step is to bound the centred moments of~$n^{-1/2}(W - nm_h)$.
From \cite[Proposition~1.3]{Chatterjee2007}, for any fixed~$k\ge1$,
\ben{
    \IE \bbbabs{\frac{W}{n}-\tanh\left( \beta \left(\frac{W}{n}\right)+h\right)}^k \leq \bigo(n^{-k/2}).
    \label{40}
}
Now, for~$y$ small enough, there exists~$C'_y < \infty$ such that
\[
     |w - \tanh(\b w + h)| \geq C'_y|w - m_h| \quad\mbox{in}\quad \ts |w-m_h| \le y.
\]
On the other hand, \cite[Theorem~1.4]{Dembo2010a} show that 
\be{
     \IP\bbbbklr{\bbbabs{\frac{W}{n}-m_h}>t} \leq e^{-n C(t)},
}
for some~$C(t)>0$, so that~$\IP(|n^{-1}W - m_h| > y) = \bigo(e^{-nC(y)})$.
Combining these last two statements, it follows from~\eq{40} that,
for any~$k\ge1$,
\ben{\label{41}
    \IE \bbbabs{\frac{W-nm_h}{\sqrt n}}^k = \bigo(1).
} 

Now we turn to verifying~\eq{10}; we use the representation in Remark~\ref{rem1}.
%Indeed,
%\ba{
% \babs{\IE[GD|S]-\IE GD}&=\frac{1}{2a}\babs{\IE [(W'-W)^2|S]-\IE (W'-W)^2} \\
%	&\leq C n \babs{\IP(\abs{W'-W}=2|S)-\IP(\abs{W'-W}=2)}.
%}
According to \cite[Lemma~4.4]{Rollin2015}, and using~\eq{41},
\ba{
   \bbbabs{\IP\bklr{{\ts W'-W=2|S}}-\frac{(1-m_h)^2}{4}}
                  &\leq C n^{-1/2} \left(\frac{\abs{W-\mu_n}}{\sigma_n}+n^{-1/2}\right), \\
   \bbbabs{\IP\bklr{{\ts W'-W=2}}-\frac{(1-m_h)^2}{4}}&\leq C n^{-1/2},
}
so that~\eq{10} is satisfied for some constant~$\k>0$, with~$k = 1$ and~$T=0$ almost surely:
\be{
     \frac{1}{a}\babs{\IP\bklr{{\ts W'-W=2|S}}-\IP\bklr{{\ts W'-W=2}}} \leq \k \sigma_n\left(\frac{\abs{W-\mu_n}}{\sigma_n}+1\right).
}
We next show~$\IE[(R')^2]=\bigo(1)$.
From~\eq{41} and~\eq{39},
\be{
   \sqrt{\IE[(R')^2]}\leq C(1+\abs{\mu_n-n m_h}).
}
To bound~$\abs{\mu_n-n m_h}$ when~$h\not=0$, note that the expectation of~\eq{36} is zero, which,
with \eq{37} and~\eq{38}, implies that
\ba{
   |nm_h-\mu_n| &\leq 
                \frac{\bigl|\sum_{i=1}^n \IE \left(\tanh\left(\frac{\beta}{n} (W-S_i)+ h\right)
                   -\tanh\left(\frac{\beta}{n} W+ h\right)\right)\bigr| 
                             +C n \IE\left(\frac{W}{n}-m_h\right)^2}{|\beta (1-m_h^2)-1|}; 
} 
applying \eq{37} and~\eq{38} yields~$\abs{\mu_n-n m_h}=\bigo(1)$, and 
hence~$\sqrt{\IE[(R')^2]}=\bigo(1)$.

Collecting the results above, it now follows from Corollary~\ref{17}\emph{(i)} that
\be{
\dtv\left(\law\left(\tW\right), 
%   \TP\left(\frac{m_h}{2},\frac{n(1-m_h^2)}{4(1-\beta+\beta m_h^2)}\right)\right)=\bigo(\s_n^{-1})=\bigo(n^{-1/2}).
    \TP\left(\half\mu_n,\quarter\s_n^2\right)\right)=\bigo(\s_n^{-1})=\bigo(n^{-1/2}).
}
For the local limit bound,  we only need to show 
\[
    S_2(\law(\tW))  = \bigo(\s_n^{-2}),
\]
which follows from \cite[Lemma~4.4]{Rollin2015}.
Noting Remark~\ref{rem1}, Corollary~\ref{17}\emph{(ii)} now easily implies that
\be{
\dloc\left(\law\left(\tW\right), 
     \TP\left(\half\mu_n,\quarter\s_n^2\right)\right)=\bigo(\s_n^{-2})=\bigo(n^{-1}).
}
Then $\mu_n$ can be replaced by~$nm_h$ and $\s_n^2$ by $\frac{n(1-m_h^2)}{(1-\beta+\beta m_h^2)}$.
This follows from properties of the translated Poisson distribution, because $|\mu_n - nm_h| = \bigo(1)$, 
$|\s_n^2 - \frac{n(1-m_h^2)}{(1-\beta+\beta m_h^2)}| = \bigo(n^{1/2})$ and $n^{-1}\s_n^2$
is bounded away from~$0$.

\section*{Acknowledgements} 

ADB and NR thank the Institute of Mathematical Sciences and the Department of Statistics and Applied Probability at the National University of Singapore for their kind hospitality. 
AR was supported by NUS Research Grant R-155-000-167-112. ADB is supported in part by
Australian Research Council Discovery Grants DP150101459 and DP150103588.
NR is supported in part by Australian Research Council Discovery Grant DP150101459. We thank the two referees for their comments.

%\begin{thebibliography}{}
%

\end{document}